\documentclass[reqno]{amsart}
\usepackage{amssymb}
\usepackage{graphicx,epsfig}
\usepackage{amsmath}
\usepackage{verbatim}
\usepackage{ulem}
\usepackage[usenames,dvipsnames]{color}
\let\<\leq
\let\>\geq

\newtheorem{thrm}{Theorem}[section]
\newtheorem{lmm}{Lemma}[section]
\newtheorem{crllr}{Corollary}[section]
\newtheorem{prpstn}{Proposition}[section]

\newtheorem{xmpl}{Example}[section]

\theoremstyle{remark}
\newtheorem{rmrk}{Remark}[section]

\newcounter{hyp}
\newenvironment{mycondition}{%
\addtocounter{equation}{-1}
\refstepcounter{hyp}

\begin{equation}}
{\end{equation}}

\newcommand{\tto}{\rightrightarrows}
\newcommand{\sgn}{{\rm sgn}\kern 0.12em}
\newcommand{\argmin}{{\rm argmin}\kern 0.12em}

\newcommand{\cC}{{\mathcal C}}

\newcommand{\cH}{{\mathcal H}}

\newcommand{\cL}{{\mathcal L}}

\newcommand{\cP}{{\mathcal{P}}}

\renewcommand{\phi}{\varphi}
\newcommand{\eps}{\varepsilon}
\newcommand{\BB}{\mathbb{B}}

\newcommand{\R}{\mathbb{R}}
\newcommand{\N}{\mathbb{N}}
\newcommand{\bd}{{\rm bd}\kern 0.12em}
\newcommand{\haus}{{\rm haus}}
\newcommand{\gph}{{\rm gph}\kern 0.12em}
\newcommand{\dom}{{\rm dom}\kern 0.12em}
\newcommand{\ran}{{\rm ran}\kern 0.12em}
\newcommand{\inte}{{\rm int}\kern 0.12em}
\newcommand{\cl}{{\rm cl}\kern 0.12em}
\newcommand{\demi}{\frac{1}{2}}
\newcommand{\ie}{{\it i.e.}\,\,}

\newcommand{\xbar}{\overline{x}}

\renewcommand{\bar}{\overline}
\newcommand{\rinf}{\R\cup\{+\infty\}}
\newcommand{\hu}{\widehat{u}}
\newcommand{\tdown}{\bigtriangledown}

\begin{document}
\title[]{Asymptotic behavior of nonautonomous monotone and subgradient evolution equations}

\author{Hedy Attouch}

\address{Institut de Math\'ematiques et Mod\'elisation de Montpellier, UMR 5149 CNRS, Universit\'e Montpellier 2, place Eug\`ene Bataillon,
34095 Montpellier cedex 5, France.}
\email{hedy.attouch@univ-montp2.fr}

\author{Alexandre Cabot}
\address{Institut de Math\'ematiques de Bourgogne, UMR 5584, CNRS, Univ. Bourgogne Franche-Comt\'e, 21000 Dijon, France. }
\email{alexandre.cabot@u-bourgogne.fr}

\author{Marc-Olivier Czarnecki}
\address{Institut de Math\'ematiques et Mod\'elisation de Montpellier, UMR 5149 CNRS, Universit\'e Montpellier 2, place Eug\`ene Bataillon,
34095 Montpellier cedex 5, France.}
\email{marco@math.univ-montp2.fr}

\date{}



\begin{abstract}
In a Hilbert setting $H$,  we study the asymptotic behavior of the trajectories of   nonautonomous evolution equations 
$\dot x(t)+A_t(x(t))\ni 0$,
where for each $t\geq 0$, $A_t:H\tto H$ denotes a maximal monotone operator. 
 We provide general conditions  guaranteeing the weak ergodic convergence of each trajectory $x(\cdot)$ to a zero of a limit maximal monotone operator $ A_\infty$,  as the time variable $t$ tends to $+\infty$. 
The crucial point is to use the Br\'ezis-Haraux function, or equivalently the Fitzpatrick function, to express at which rate  the excess of  $\gph A_\infty$ over $\gph A_t$ tends to zero.
This approach gives a sharp and unifying view on this subject.
In the case of  operators $A_t= \partial \phi_t$ which are subdifferentials of closed convex  functions $\phi_t$, we show convergence results for the trajectories.
Then, we  specialize our results to  multiscale evolution equations, and obtain asymptotic properties of hierarchical minimization, and  selection of viscosity solutions.
Illustrations are given in the field of coupled systems, and partial differential equations.
\end{abstract}

\thanks{H. Attouch, A. Cabot, and M.-O. Czarnecki: with the support of  ECOS  grant C13E03.\\
H. Attouch: Effort sponsored by the Air Force Office of Scientific Research, Air Force Material Command, USAF, under grant number FA9550-14-1-0056.}
\maketitle

\section{Introduction and notations}
 Throughout the paper, $H$ is a real Hilbert space which is endowed with the scalar product $\langle \cdot,\cdot\rangle$ and the norm $\|\cdot\|$ defined by $\|x\|= \sqrt{\langle x,x\rangle}$ for any $x\in H$. We study the asymptotic behavior of the  NonAutonomous Monotone Inclusion
\begin{equation}\label{eq.NAMI}\tag{NAMI}
\dot x(t)+A_t(x(t))\ni 0, \qquad t\geq 0,
\end{equation}
where for every $t\geq 0$, $A_t:H\tto H$ denotes a maximal monotone operator. Following Br\'ezis \cite[Definition 3.1]{Bre}, we say that $x:[0,+\infty[\to H$ is a strong global solution  of \eqref{eq.NAMI}
if $x(\cdot)$ is locally absolutely continuous on $[0,+\infty[$, and  if \eqref{eq.NAMI} holds for almost all $t>0$.
We take for granted the existence of strong solutions to \eqref{eq.NAMI}.
The existence of solutions of nonautonomous differential inclusions governed by time dependent maximal monotone operators is a nontrivial topic. 
This issue has been studied extensively in the years 70-80, see  Br\'ezis \cite{Bre}, Attouch and Damlamian \cite{AD}, Kenmochi \cite{Ken}, and references therein.

We prove the ergodic weak convergence of the trajectories of \eqref{eq.NAMI}
under some general condition involving the Br\'ezis-Haraux function associated to the operator~$A_t$. The Br\'ezis-Haraux function $G_M:H\times H\to \rinf$ associated to the maximal monotone operator $M$ was introduced in \cite{BreHar}. It is defined by
$$G_M(x,u)=\sup_{(y,v)\in \gph M} \langle x-y, v-u\rangle,$$
where $\gph M$ denotes the graph of $M$.
The function $G_M$ is nonnegative and takes the zero value on the graph of $M$.
The function $G_M$ is connected with the Fitzpatrick function $F_M$ via the formula $G_M(x,u)=F_M(x,u)-\langle x,u\rangle$, for every $(x,u)\in H\times H$. If there exists a maximal monotone operator $A_\infty: H\tto H$ such that $S_\infty=A_\infty^{-1}(0)\neq \emptyset$, and if
$$\forall (z,p)\in \gph A_\infty, \qquad \int_0^{+\infty}G_{A_t}(z,p)\, dt<+\infty,$$
we show that every strong global solution of \eqref{eq.NAMI}  converges weakly in average toward an element of $S_\infty$,  as $t\to +\infty$. As a by-product, we recover the Baillon-Br\'ezis theorem \cite{BaiBre} in the case of an autonomous evolution inclusion. The above integral condition is well suited for structured problems of the form $A_t=A+\beta(t)B$, with $A$, $B:H\tto H$ maximal monotone operators, and $\beta(t)$ a time-dependent parameter. In this framework, we recover as a particular case a condition due to Bot-Csetnek \cite[Section~2]{BotCse} that guarantees the weak ergodic convergence of a forward-backward penalty scheme. The Bot-Csetnek condition formulated by means of the Fitzpatrick function is itself a generalization of a former condition given by Attouch-Czarnecki \cite{AttCza}, see also \cite{AttCzaPey1,AttCzaPey2}.

The second important part of the paper concerns the study of the asymptotic behavior of the  NonAutonomous subGradient Inclusion
\begin{equation}\tag{NAGI}
\dot x(t)+\partial \phi_t(x(t))\ni 0, \qquad t\geq 0,
\end{equation}
where for every $t\geq 0$, $\phi_t:H\to \rinf$ is a closed convex function. Such an evolution inclusion falls into the framework of \eqref{eq.NAMI} since the operator $\partial \phi_t:~H\tto H$ is maximal monotone. In the context of subdifferential operators, we can obtain convergence of the trajectories instead of ergodic convergence. If we assume that the filtered family $(\phi_t)_{t\geq 0}$ is nonincreasing with respect to $t$, then we easily show that the potential energy function $t\mapsto\phi_t(x(t))$ decreases toward its infimum as $t\to +\infty$. By using the Opial lemma along with a suitable summability condition, we deduce the weak convergence of the trajectories, see Theorem \ref{th.nonincreasing}. When no monotonicity assumption is made on the family $(\phi_t)_{t\geq 0}$, it may be tricky to prove that $\lim_{t\to +\infty}\phi_t(x(t))$ exists. The reader is referred to  \cite{FurMiyKen}, where  ad hoc conditions are given in order to control the variations
 in time of the family $(\phi_t)_{t\geq 0}$. Weak convergence of the trajectories is then obtained via energetical arguments. In the present paper, we propose an alternative approach, based on the study of the distance from the trajectory to the optimal set\footnote{The optimal set $S_\infty$ is the set of minimizers (supposed to be nonempty) of the function $\phi_\infty$, that is obtained as the limit of $\phi_t$ as $t\to +\infty$ (in a sense to be precised).} $S_\infty$. The argument follows from an extension of a result due to Baillon-Cominetti \cite{BaiCom} in a finite dimensional framework. Under a suitable summability assumption, we derive the weak convergence of every  trajectory of \eqref{eq.NAGI} toward a point of the optimal set $S_\infty$, see Theorem \ref{th.NAGI}. 

A particular attention is devoted to the case $\phi_t=\Phi+\beta(t) \Psi$, where $\Phi$, $\Psi:H\to \rinf$
are closed convex functions, and $\beta(t)$ is a positive time-dependent parameter. This corresponds to the situation of coupled (sub)gradients with multiscale aspects. If $\beta(t)\to +\infty$, and if the set $C=\argmin \Psi$ is nonempty, 
the orbits of the Multiscale Asymptotic Gradient dynamics, studied in \cite{AttCza},
\begin{equation}\label{MAG}\tag{MAG}
\dot x(t)+\partial \Phi(x(t))+\beta(t)\partial \Psi(x(t))\ni 0
\end{equation}
tend to minimize the function $\Phi$ over the set $\argmin \Psi$, thus leading to a hierarchical minimization process.
The problem of convergence as $t\to +\infty$ depends on the behavior as $\eps\to 0$ of the quantity $\omega(\eps)$ defined by
$$\omega(\eps)=\inf_H\left(\left(\Psi-\inf_H\Psi\right)+\eps\left(\Phi-\inf_C\Phi\right)\right). $$
The key condition that implies weak convergence of the trajectories of
\eqref{MAG} is the following
$$\int_0^{+\infty}\beta(t)\left|\omega (1/\beta(t))\right|\, dt<+\infty.$$
The map $\omega (\cdot)$ was introduced by Cabot \cite{Cab} in the framework of a diagonal proximal point algorithm involving multiscale aspects. The behavior of the map $\omega (\cdot)$ was used later by 
Alvarez-Cabot \cite{AlvCab} to find asymptotic selection properties of viscosity equilibria for semilinear evolution equations. By resorting to the duality theory, we show that the quantity $|\omega(\eps)|$ is majorized by an expression depending only on the function $\Psi$. More precisely, there exists $p\in H$ in the range of the normal cone operator $N_C:H\tto H$, such that\footnote{The functions $\Psi^*$ and $\sigma_C$ denote respectively the Fenchel conjugate of $\Psi$ and the support function of $C$. }
$$|\omega(\eps)|\leq \Psi^*(\eps p)+\min_H \Psi-\sigma_C(\eps p),$$
for every $\eps\geq 0$.
Assuming that $\min_H \Psi=0$, we deduce that the above summability condition is satisfied as soon as
$$ \displaystyle \int_{0}^{+\infty} \beta (t) \left[\Psi^* \left(\frac{p}{ \beta (t)}\right) - \sigma_C \left(\frac{p}{ \beta (t)}\right)\right]dt < + \infty,
$$
for every vector $p$ in the range of $N_C$.
This is precisely the condition due to Attouch-Czarnecki \cite{AttCza} in order to ensure weak convergence of the trajectories of \eqref{MAG}. When the function $\Psi$ satisfies the quadratic conditioning property $\Psi\geq a\,d^2(\cdot,C)$ for some $a>0$, the above assumption is fulfilled if $\int_0^{+\infty}(1/\beta(t))\, dt <+\infty.$

Each of the above mentioned convergence results relies on a summability condition with respect to some suitable quantity.
The summability condition expresses that the integrand tends to zero sufficiently fast. Therefore the conditions stated above quantify the fact that the operators $A_t$ (resp. functions $\phi_t$) tend sufficiently \textit{fast} toward
their limit $A_\infty$ (resp. $\phi_\infty$).

 The problem of trajectory convergence toward a particular viscosity solution naturally arises when the operators $A_t$ (resp. functions $\phi_t$) \textit{slowly} tend toward their limit. We give an answer to this important issue in  two cases:

i) A first answer is given for a family $(\phi_t)_{t\geq 0}$ of closed convex functions by using a technique of central path. 
For every $t\geq 0$, we assume that the function $\phi_t$ has a strong minimum $\xi(t)\in H$,
i.e., for all $x\in H$ 
$$\phi_t (x)\geq \phi_t(\xi(t))+\alpha(t)\, \|x-\xi(t)\|^2,
\quad \mbox{for some }\alpha(t)>0.$$
Under the slow condition $\int_0^{+\infty}\alpha(t)\, dt=+\infty$, we show that any solution $x(\cdot)$ of \eqref{eq.NAGI} satisfies $\lim_{t\to +\infty}\|x(t)-\xi(t)\|=0$, thus it is attracted toward the optimal path $\xi(\cdot)$. It ensues that the trajectory $x(.)$ strongly converges if and only if
the optimal path has a limit as $t\to +\infty$, and in this case the limits are equal. The phenomenon of attraction toward the central path was brought to light in \cite{AttCom}, under a strong convexity property.

\smallskip

ii) A second answer is given in the case of the multiscaled evolution system
\begin{equation}\label{eq.viscosity}\tag{MAG$_\eps$}
\dot x(t)+\partial \Phi(x(t))+\eps(t)\, \partial \Psi(x(t))\ni 0,
\end{equation}
where $\eps:\R_+\to \R_+^*$ is a slowly vanishing viscosity coefficient, i.e., $\lim_{t\to+\infty}\eps(t)=0$ and $\int_0^{+\infty}\eps(t)\, dt=+\infty$.
By reversing the roles of the functions $\Phi$ and $\Psi$, and by using a suitable time rescaling, which allows to pass from $\beta(t) \to + \infty$ to $\eps(t) \to 0$, we show the convergence of the trajectories of \eqref{eq.viscosity} to particular solutions.  As an important special case, if the set $\argmin_C\Psi$ is a singleton $\{\xbar\}$ for some $\xbar\in H$ (where $C= \argmin \Phi$), then for any strong global solution $x(\cdot)$ of \eqref{eq.viscosity}, we have $x(t)\to \xbar$ strongly in $H$ as $t\to +\infty$.
In the case of the Tikhonov approximation $\Psi(x)= \|x\|^2$,  we obtain strong convergence  to the element of minimal norm. Note that we do not assume $\eps(\cdot)$ to be nonincreasing. Under such general assumption,  this asymptotic selection result for the Tikhonov approximation was first obtained  by Cominetti-Peypouquet-Sorin \cite{ComPeySor}.

The paper is organized as follows. In Section \ref{se.NAMI}, we study the asymptotic behavior of the strong global solutions of \eqref{eq.NAMI}. The main result gives the ergodic weak convergence of the trajectories under some general condition involving the Br\'ezis-Haraux function. Section \ref{se.NAGI} is devoted to the case  $A_t=\partial \phi_t$ for a family $(\phi_t)_{t\geq 0}$ of closed convex functions. In this framework, we show weak convergence of the trajectories, thus making more precise the results of Section \ref{se.NAMI}. A special attention is dedicated to the case of structured problems of the form $\phi_t=\Phi+\beta(t)\Psi$, where $\Phi$, $\Psi: H\to \rinf$ are closed convex functions and $\beta(t)$ is a parameter tending to infinity as $t\to +\infty$. For these problems, a key ingredient consists in the study of the infimum value associated to the viscosity minimization problem $\inf_H(\Psi+\eps\Phi)$. Section \ref{se.omega} is devoted to this question, wi!
 th new results obtained by using duality arguments.
 Symetrically, we consider the case $\phi_t=\Phi+\eps(t)\Psi$, where $\eps:\R_+\to \R_+^*$ is a slowly vanishing viscosity coefficient.
We complete this study by considering two other classes of nonautonomous subgradient inclusions, corresponding respectively to the quasi-autonomous case, and the sweeping process. Illustrations of our results in the case of coupled gradient systems with multiscale aspects are given in Section~\ref{se.examples}.

\bigskip

\noindent{\bf Notations.} For a function $f:H\to \rinf$, the set $\dom f=\{x\in H: \, f(x)<+\infty\}$ is called the domain of $f$. We call $f$ a proper function if $\dom f$ is a nonempty set.
Let $f: H \to \R\cup\{+\infty\}$ be a proper convex function. The subdifferential of $f$ at $x\in \dom f$
is defined by
$$\partial f(x)=\{p\in H:\, f(y)\geq f(x)+\langle p,y-x\rangle\quad\forall y\in H\}.$$
If the function $f$ is closed and convex, the multivalued operator $\partial f: H\tto H$ is maximal monotone.
For a nonempty convex set $C\subset H$, the normal cone to $C$ at $x\in C$ is given by
$$N_{C}(x)=\{p\in H:\langle p,y-x\rangle\<0\quad\forall y\in C\}.$$
It coincides with the set $\partial \delta_C(x)$, where $\delta_C$ is the indicator function of $C$, taking the value $0$ on $C$, and $+\infty$ elsewhere.
 The Fenchel conjugate of a function $f: H \to \R\cup\{+\infty\}$ is defined by 
$\, f^*(p)=\sup_{x\in H}\{\langle p,x\rangle- f(x)\}$ for every $p\in H$.
The support function of the set $C\subset H$ is given by \,$\sigma_{C}(p)=\delta_C^*(p)=\sup_{x\in C}\langle p,x\rangle$ for every $p\in H$.
Given two functions $f$, $g:H\to \rinf$, we define the inf-convolution of $f$ and $g$ as follows: for every $x\in H$, 
$$(f\tdown g)(x)=\inf_{y\in H}\left\lbrace f(y)+g(x-y)\right\rbrace .$$
Recall that the equality $(f\tdown g)^*=f^*+g^*$ is always true, while the equality $(f+g)^*=f^*\tdown g^*$ holds true if $f$, $g$ are closed convex, and if there exists $x_0\in \dom f$ such that $g$ is continuous at $x_0$. This last condition is known as the Moreau-Rockafellar condition.
For classical facts on convex analysis, see for example \cite{AttButMic,Aze,BC,EkeTem,Roc,RocWet}.

\section{Nonautonomous monotone inclusion}\label{se.NAMI}
In our approach, the Br\'ezis-Haraux and the Fitzpatrick functions will play a crucial role in order to capture the asymptotic behaviour of the filtered sequence of maximal monotone operators 
$(A_t)_{t\to +\infty}$.

\subsection{Graph convergence and convergence of  the Br\'ezis-Haraux functions}
A set-valued mapping $M$ from $H$ to $H$ assigns to each $x\in H$ a set $M(x)\subset H$, hence it is a mapping from $H$ to $2^H$. 
Every set-valued mappping $M:H\to 2^H$ can be identified with its graph defined by $$\gph M=\{(x,u)\in H\times H: \, u\in Mx\}.$$
To emphasize this, we speak of $M$ as a multivalued operator (or multifunction, or correspondence) and we write $M:H\tto H$.
The domain and range of $M:H\tto H$ are taken to be the sets
\begin{center}
$\dom M=\{x\in H: \ \exists u\in H \mbox{ with } u\in Mx\},$\\
\vspace{2mm}
$\ran (M)=\{u\in H: \ \exists x\in H \mbox{ with } u\in Mx\}.$
\end{center}
An operator $M:H\tto H$ is said to be monotone if for any $(x,u)$, $(y,v)\in \gph M$, one has
$\langle y-x, v-u\rangle \geq 0$. It is maximal monotone if there exists no monotone operator whose graph strictly contains $\gph M$. 
For classical facts on maximal monotone operators in Hilbert spaces, see for example \cite{BC,RocWet}.
Given a maximal monotone operator $M$, the Br\'ezis-Haraux function 
$G_M:H\times H\to \rinf$, introduced in \cite{BreHar}, is defined by
$$G_M(x,u)=\sup_{(y,v)\in \gph M} \langle x-y, v-u\rangle.$$
Let us show that $G_M$ is an exterior penalty function with respect to the graph of $M$.
By Minty's theorem, we have the following characterization of  
$(x,u)\in \gph M$
\begin{align*}
u \in Mx & \Leftrightarrow x +u \in x +Mx  \\
         & \Leftrightarrow x = (I+M)^{-1}(x+u) \\
& \Leftrightarrow x - (I+M)^{-1}(x+u) =0.
\end{align*}
Thus, the function 
$$P_M (x,u) := \|x - (I+M)^{-1}(x+u) \|^2$$
is a  penalty function with respect to the graph of $M$. It is nonnegative, Lipschitz continuous on bounded sets, and $P_M (x,u) =0 \Leftrightarrow  (x,u)\in \gph M$.
But $P_M$ is difficult to handle practically because, in general, the computation of  the resolvent is a difficult task.
Let us show that the Br\'ezis-Haraux function solves some of these difficulties.
Given arbitrary $(x,u)\in H \times H$, by Minty's theorem, there exists a unique $y\in H$ such that
$$
y + My \ni x+u,
$$
which is $y= (I+M)^{-1}(x+u)$.
Set $v= x+u -y$. We have $v \in My$, and $v-u= x-y$. Thus
\begin{align}
G_M(x,u)=&\sup_{(\xi,\eta)\in\gph M} \langle x-\xi,\eta-u\rangle\nonumber\\
         & \geq \langle x-y,v-u\rangle \nonumber\\
& = \|x-y\|^2  \nonumber\\
& = \|x-(I+M)^{-1}(x+u)\|^2 = P_M (x,u)\label{basic-ineq4}.
\end{align}
On the other, by monotonicity of $M$, we immediately have that $G_M$ is less or equal than zero on the graph of $M$.
Thus $G_M$ is an exterior penalty function with respect the graph of $M$, see also \cite[Corollary 3.9]{Fitz}. A major advantage of $G_M$ is that it is more flexible than $P_M$ for the practical computation, as we will show later. Another interesting feature of $G_M$ is its close relationship with the convex analysis.\\
 The Fitzpatrick function $F_M:H\times H\to \rinf$  is defined by
$$F_M(x,u)=\sup_{(y,v)\in \gph M} \lbrace\langle x, v\rangle+\langle y,u\rangle-\langle y,v\rangle\rbrace.$$
The function $F_M$ was introduced by Fitzpatrick in \cite{Fitz}. 
As a supremum of continuous affine functions,  $F_M$ is convex and lower semicontinuous with respect to the couple $(x,u)$. This property makes it an effective tool to address the problems governed by maximal monotone operators, using methods of convex analysis.
It is the subject of active research, see
for example \cite{BauLarSen,BurSva,MarSva,MarThe,PenZal,SimZal1,SimZal2}. 
The function $G_M$ is related to $F_M$ by $$G_M(x,u)=F_M(x,u)-\langle x,u\rangle.$$ 


The convergence of nets of maximal monotone operators can be formulated in terms of the Br\'ezis-Haraux function.

\begin{prpstn} \label{pr.resolvent convergence}
 Let $\{A_t:H\tto H, \, t\geq 0\}$ be a family of maximal monotone operators. Assume that there exists a maximal monotone operator $A_\infty: H\tto H$ such that
$$
\forall (z,p)\in \gph A_\infty, \qquad \lim_{t \to+\infty}G_{A_t}(z,p)=0.
$$
Then, $(A_t )$ converges in the resolvent sense to $A_\infty$. Equivalently, $(A_t)$ graph converges to $A_\infty$.
\end{prpstn}
\begin{proof}
Take arbitrary $y \in H$. By Minty's theorem there exists a unique 
$z \in H$ such that $z + A_\infty z \ni y$.
Set $p= y-z$, we have $p \in A_\infty z$, and $z = (I + A_\infty)^{-1}y$.
By (\ref{basic-ineq4})
\begin{align}
G_{A_t}(z,p) &\geq  \|z-(I+A_t)^{-1}(z+p)\|^2 \nonumber\\
&=\| (I + A_\infty)^{-1}y -(I+A_t)^{-1}y \|^2\label{resol-conv2}.
\end{align}
By assumption, $\lim_{t \to+\infty}G_{A_t}(z,p)=0$, which, by (\ref{resol-conv2}), implies the convergence of the resolvents.
Recall that, for sequence of maximal monotone operators, the convergence of the resolvents is equivalent to the graph convergence, \cite[Proposition 3.60]{Att}. 
\end{proof}

\begin{rmrk}
The main ingredient in the previous result is the inequality $G_M \geq P_M$, that already appears in a paper by Penot \& Zalinescu, see 
\cite[Lemma 2.3]{PenZal}. By using the same inequality, it is shown in \cite[Proposition 3.1]{PenZal}
that if $G_{A_t}$ converges to $G_{A_\infty}$ in the bounded-Hausdorff sense, then 
$A_t \to A_\infty$ for the bounded-Hausdorff convergence.
\end{rmrk}

The following example shows that the  convergence of the Br\'ezis-Haraux functions (equivalently, of  the Fitzpatrick functions), is a stronger notion of convergence than the graph convergence.

Take $A$ a general maximal monotone operator, and  $\eps :\R_+\to H$ a map such that $\lim_{t \to +\infty}\eps (t)=0$.
Set $A_t (x)= A(x) + \eps (t) $, with $\dom A_t = \dom A $.
It is immediate to verify that $A_t$ is maximal monotone, and $A_t$ graph-converges to $A$ as $t \to +\infty$. An elementary computation gives, for any $(x,u)\in H \times H$
$$
G_{A_t}(x,u)= G_A (x, u -\eps (t)).
$$
Therefore, to obtain the convergence of graphs without convergence of the Br\'ezis-Haraux functions, it is sufficient to produce a maximal monotone operator $A$ such that
$$
u \mapsto G_A (x, u)
$$
is not continuous at a point $(x,u)\in \gph A$. Since they differ by a continuous bilinear term ($G_A(x,u)=F_A(x,u)-\langle x,u\rangle$), 
it is equivalent to prove the result for the mapping $u \mapsto F_A (x, u)$.
Let us specialize  $A \in \mathcal B (H)$ to be a bounded linear monotone self-adjoint operator. Let 
$q_A : H \to \mathbb R$, $q_A (x) = \frac{1}{2}\langle x,Ax\rangle$ be the quadratic form associated to $A$. 
By a straight computation using the Fenchel conjugate, see
 \cite[Example 20.45]{BC}
$$
F_A(x,u) = 2 (q_A)^*  \big(\frac{1}{2}u + \frac{1}{2} Ax \big).
$$
As a consequence, it is sufficient to consider $A$ such that $(q_A)^*$ is not continuous.
This means that $A$ is not invertible (it is only  positive semi-definite).
For example,  when $A=0$, then $F_A$ is the indicator function of $H \times \{0\}$, an extreme situation where the continuity property of $u \mapsto F_A (x, u)$ fails to be satisfied.  Remark that, if $A \in \mathcal B (H)$  is strongly monotone, then $(q_A)^*$ is continuous, and  the two notions of convergence  coincide (in that particular case).

\subsection{Nonautonomous monotone inclusion: Ergodic convergence}

In this section, we study  the asymptotic behavior of the trajectories of
\begin{equation}
\tag{\ref{eq.NAMI}}
\dot x(t)+A_t(x(t))\ni 0, \qquad t\geq 0.
\end{equation}

The trajectory $x(\cdot)$ is a strong global solution  of \eqref{eq.NAMI} in the sense of Br\'ezis \cite[Definition 3.1]{Bre}, i.e., $x : [0,+\infty[\to H$ is absolutely continuous on any bounded interval $[0,T]$, and  \eqref{eq.NAMI} holds for almost every $t>0$.

Recall that an absolutely continuous function is differentiable almost everywhere, and that one can recover the function from its derivative by the usual integration formula.
Uniqueness of the solution for a given Cauchy data is an immediate consequence of the monotonicity of the operators $A_t$. In the sequel, we take for granted the existence of strong solutions to \eqref{eq.NAMI}.

\subsubsection{Statement of the  ergodic convergence result}
\begin{thrm}\label{th.ergodic_weak_cv}
Let $\{A_t:H\tto H, \, t\geq 0\}$ be a family of maximal monotone operators. Assume that there exists a maximal monotone operator $A_\infty: H\tto H$ such that $S_\infty=A_\infty^{-1}(0)\neq \emptyset$ and
\begin{mycondition}\label{eq.cond_Brezis-Haraux}
\forall (z,p)\in \gph A_\infty, \qquad \int_0^{+\infty}G_{A_t}(z,p)\, dt<+\infty.
\end{mycondition}
Then  every strong global solution $x(.)$ of \eqref{eq.NAMI} converges weakly in average to some $x_\infty\in S_\infty$, i.e., as  $t\to +\infty$, 
$$
\frac{1}{t} \int_0^t x(s)\, ds\rightharpoonup~x_\infty.
$$
\end{thrm}

\begin{rmrk} 
From (\ref{resol-conv2}), we deduce that Condition (\ref{eq.cond_Brezis-Haraux}) implies 
\begin{equation}\label{eq.ess_graph_conv}
\forall y \in H, \quad \int_0^{+\infty}\| (I+A_t)^{-1}y -(I + A_\infty)^{-1}y \|^2 \, dt<+\infty.
\end{equation}
Hence, for all $y \in H$
\begin{equation}\label{eq.ess_graph_conv2}
\liminf {\rm ess}_{t \to + \infty}\| (I+A_t)^{-1}y -(I + A_\infty)^{-1}y \|=0,
\end{equation}
a property which is directly related to the graph convergence of $A_t$ to $A_\infty$,  as  $t\to +\infty$ (recall that the graph convergence of a filtered sequence of maximal monotone operator is equivalent to the pointwise convergence of the resolvents).
The detailed study of this relationship is an interesting subject for further research. 
Let us just say that, when $H$ is separable, a thorough inspection of properties (\ref{eq.ess_graph_conv}) and (\ref{eq.ess_graph_conv2}), combined with the non expansive property of the resolvents, is likely to provide (up to a negligeable set) the graph convergence of $A_t$ to $A_\infty$. \\
Indeed, it is not necessary  to deepen this topological analysis, as for our purpose, the integral form (\ref{eq.cond_Brezis-Haraux}), which is used throughout  this paper, is a more convenient  way to express the convergence of $A_t$ to $A_\infty$.  It carries more information than the topological one: it
expresses that, in the sense of the Br\'ezis-Haraux functions, the excess of $\gph A_\infty$ over $\gph A_t$ tends to $0$ fast enough as $t\to +\infty$.

\end{rmrk}

As a special case of Theorem \ref{th.ergodic_weak_cv},  we recover Baillon-Br\'ezis theorem~\cite{BaiBre}.

\begin{crllr}\label{Cor-BaiBre} \cite{BaiBre} Let $A: H\tto H$ be a maximal monotone operator such that $S=A^{-1}(0)\neq \emptyset$. Let $x(\cdot)$ be a strong global solution of $$\dot x(t)+A(x(t))\ni 0.$$ Then there exists $x_\infty\in A^{-1}(0) $ such that $\frac{1}{t} \int_0^t x(s)\, ds\rightharpoonup~x_\infty$ weakly in $H$, as $t\to +\infty$.
\end{crllr} 
\begin{proof} Take $A_t=A$ for every $t\geq 0$, and $A_\infty=A$.
Since $G_A(z,p)=0$ for every $(z,p)\in \gph A$, Condition (\ref{eq.cond_Brezis-Haraux}) is verified, and therefore Theorem~\ref{th.ergodic_weak_cv} applies.
\end{proof}

\subsubsection{Proof of Theorem~\ref{th.ergodic_weak_cv}}

Let us recall the Opial lemma~\cite{Opi}, along with an ergodic version named the Opial-Passty lemma.
\begin{lmm}[Opial] \label{lm.Opial}
Let $H$ be a Hilbert space and $x : [0,+\infty[ \to H$ be
a function such that there exists a nonempty set $S \subset H$ which verifies 
\begin{itemize}
\item[$(i)$] $\forall z \in S$, $\lim_{t \to +\infty} \|x(t)-z\|$ exists.
\item[$(ii)$] $\forall t_n \to +\infty$ with $x(t_n) \rightharpoonup x_\infty$ weakly in $H$,
we have $x_\infty \in S$.
\end{itemize}
Then, $x(t)$ converges  weakly as $t \to +\infty$ to some element $x_\infty$ of $S$.
\end{lmm}
For the following ergodic variant of the Opial lemma, the reader is referred to~\cite{Pas}.
\begin{lmm}[Opial-Passty] \label{lm.Opial-Passty}
Let $H$ be a Hilbert space, let $S$ be a nonempty subset of $H$ and let $x : [0,+\infty[ \to H$ be a function.  For any $t>0$ set $X(t)=\frac{1}{t} \int_0^t x(s)\ ds$, and assume that 
\begin{itemize}
\item[$(i)$] $\forall z \in S$, $\lim_{t \to +\infty} \|x(t)-z\|$ exists.
\item[$(ii)$] $\forall t_n \to +\infty$ with $X(t_n) \rightharpoonup X_\infty$ weakly in $H$,
we have $X_\infty \in S$.
\end{itemize}
Then, $X(t)$ converges weakly  as $t \to +\infty$ to some element $X_\infty$ of $S$.
\end{lmm}

The proof of Theorem~\ref{th.ergodic_weak_cv} relies on the Opial-Passty lemma applied with $S_\infty=A_\infty^{-1}(0)$. Let us first show that for every $z\in S_\infty$, $\lim_{t\to +\infty} \|x(t)-z\|$ exists. Fix $z\in S_\infty$ and set $h(t)=\demi \|x(t)-z\|^2$. Since $-\dot x(t)\in A_t(x(t))$ for a.e. $t\in \R_+$, we have
$$\dot h(t)=\langle x(t)-z, \dot x(t)\rangle\leq G_{A_t}(z,0)\quad \mbox{a.e. on }\R_+.$$
From this inequality and assumption (\ref{eq.cond_Brezis-Haraux}) at the point $(z,0)$, it follows that $\dot h_+\in L^1(0,+\infty)$.
From a classical lemma, this implies that $\lim_{t\to +\infty} h(t)$ exists in $\R$. Let us now show that every sequential weak cluster point   
 of $X(t)=\frac{1}{t}\int_0^t x(s)\, ds$   belongs to $S_\infty$. Let $(z,p)\in \gph A_\infty$, and consider again the function $h$ defined by 
$h(t)=\demi \|x(t)-z\|^2$. Since $-\dot x(t)\in A_t(x(t))$ for a.e. $t\in \R_+$, we obtain
$$\dot h(t)+\langle x(t)-z,p\rangle=\langle x(t)-z, p+\dot x(t)\rangle\leq G_{A_t}(z,p)\quad \mbox{a.e. on }\R_+.$$
By integrating on $[0,t]$, we find
$$h(t)+\left\langle \int_0^t x(s)\, ds-tz,p\right\rangle\leq h(0)+\int_0^t G_{A_s}(z,p)\, ds.$$
After division by $t$, and taking into account of $h(t)\geq 0$, we have
\begin{eqnarray*}
\langle X(t)-z,p\rangle&\leq& \frac{1}{t} h(0)+\frac{1}{t}\int_0^t G_{A_s}(z,p)\, ds\\
&\leq& \frac{c}{t} \quad \mbox{ with } c=h(0)+\int_0^{+\infty} G_{A_s}(z,p)\, ds.
\end{eqnarray*}
Suppose now that $X(t_n)\rightharpoonup X_\infty$ as $n\to +\infty$ for a sequence $t_n\to +\infty$. Taking the limit as $n\to +\infty$ in
$\langle X(t_n)-z,p\rangle \leq c/t_n$, we immediately obtain $\langle X_\infty-z,p\rangle \leq 0$. Hence we have proved that for every $(z,p)\in \gph A_\infty$,
$$\langle X_\infty-z, 0-p\rangle \geq 0.$$
The maximal monotonicity of $A_\infty$ allows us to infer  that $0\in A_\infty(X_\infty)$, that is $X_\infty\in S_\infty$. By Lemma \ref{lm.Opial-Passty},
we conclude to the weak ergodic convergence of the trajectories of \eqref{eq.NAMI}. $ \hfill  \square $
%
%
%
%
\subsection{Coupled operators with multiscale aspects: $A_t=A+\beta(t) B$ with $\beta(t) \to~+\infty$}

In this section, we specify our general ergodic convergence result to the case of a structured operator of the form $A_t=A+\beta(t) B$. The parameter $\beta(t)$ is assumed to tend to $+\infty$, thus leading to a two-scale problem.

 \begin{thrm}\label{co.weak_ergo_cv:A+betaB}
 Let $A$, $B: H\tto H$ be two maximal monotone operators such that the sets $C=B^{-1}(0)$ and $\left(A+N_{C}\right)^{-1}(0)$ are nonempty. Assume that
 the operator $A+N_C$ is maximal monotone. Given a map $\beta:\R_+\to \R_+^*$, assume that the operator $A+\beta (t) B$ is maximal monotone for every $t\geq 0$. Suppose additionally that
 \begin{mycondition}\label{eq.Brezis-Haraux:A+betaB} 
 \forall z\in C, \quad \forall q\in N_C(z), \qquad \int_0^{+\infty} \beta(t)G_B\left(z,\frac{q}{\beta(t)}\right)\, dt<+\infty.
 \end{mycondition}
 Then every strong global solution $x(.)$ of the Multiscale Asymptotic Monotone Inclusion
\begin{equation}\label{MAMI}\tag{MAMI}
\dot x(t)+A(x(t))+\beta(t)\, B(x(t))\ni 0,
\end{equation}
converges weakly in average to some $x_\infty\in (A+N_C)^{-1}(0)$, i.e., as  $t\to +\infty$, 
 $$\frac{1}{t} \int_0^t x(s)\, ds\rightharpoonup~x_\infty.
$$
 \end{thrm}

 \begin{rmrk} A particularly (new) interesting situation covered by the above theorem  is the  case $\beta(t) \to +\infty$. Indeed, a quick formal inspection of the formula (\ref{eq.Brezis-Haraux:A+betaB}) shows that, if $\beta(t)$ tends to a finite value, then $B= N_C$, a  situation where the classical ergodic convergence  theorem of Baillon-Br\'ezis can be applied.
\end{rmrk}

\begin{rmrk} Denoting by $F_B$ the Fitzpatrick function associated to the operator~$B$, we have for every $q\in N_C(z)$,
 \begin{eqnarray*}
 G_B\left(z,\frac{q}{\beta(t)}\right)&=&F_B\left(z,\frac{q}{\beta(t)}\right)-\left\langle z,\frac{q}{\beta(t)}\right\rangle\\
 &=&F_B\left(z,\frac{q}{\beta(t)}\right)-\sigma_C \left(\frac{q}{\beta(t)}\right).
 \end{eqnarray*}
 The last equality is an immediate consequence of the Fenchel extremality relation $\delta_C(z)+\sigma_C(q)=\langle z,q\rangle$. It ensues
 that condition (\ref{eq.Brezis-Haraux:A+betaB}) can be equivalently rewritten as
 \begin{mycondition}
\label{eq.Fitzpatrick:A+betaB} 
 \forall z\in C, \, \forall q\in N_C(z), \quad \int_0^{+\infty} \beta(t)\left[ F_B\left(z,\frac{q}{\beta(t)}\right)-\sigma_C \left(\frac{q}{\beta(t)}\right) \right]\, dt<+\infty.
 \end{mycondition}
 This last condition was recently introduced in the discrete setting by Bot \& Csetnek \cite{BotCse}  as a generalization of Condition (\ref{eq.Attouch-Czarnecki}) below and its discrete counterpart. \end{rmrk}

As a consequence of Theorem~\ref{co.weak_ergo_cv:A+betaB}, we recover the ergodic convergence result of Attouch \& Czarnecki \cite{AttCza}.

\begin{crllr} \cite[Theorem 2.1, (i)]{AttCza}
Let $A: H \tto  H$  be a  maximal monotone operator, let $\Psi:  H \rightarrow \R_+ \cup \{+ \infty\}$ be a
closed convex proper function, such that $C$ = $\argmin\Psi = \Psi ^{-1}(0) \neq \emptyset$, let $\beta : \R_+ \rightarrow \R_+$ be  a  measurable function. Assume that $A + N_C$ is a maximal monotone operator and $S:= (A + N_C)^{-1}(0)$ is  non empty, and
\begin{mycondition}
\label{eq.Attouch-Czarnecki}
 \forall p\in \ran(N_C), \quad \displaystyle \int_{0}^{+\infty} \beta (t) \left[\Psi^* \left(\frac{p}{ \beta (t)}\right) - \sigma_C \left(\frac{p}{ \beta (t)}\right)\right]dt < + \infty.
\end{mycondition}
 Then, for every strong global solution trajectory $x(.)$ of the differential
 inclusion
\begin{equation}\tag{\ref{MAG}}
\dot x(t)+A(x(t))+\beta(t)\, \partial\Psi(x(t))\ni~0
\end{equation}
there exists $ x_{\infty}\in S  $  such that 
$$ \  w-\lim_{t\rightarrow
  +\infty}\mbox{  } \frac{1}{t} \int_0^t x(s) ds = x_\infty.$$
\end{crllr}

Indeed, apply  Theorem~\ref{co.weak_ergo_cv:A+betaB} with $B=\partial \Psi$. 
 Recalling that
 $$F_{\partial \Psi}\left(z,\frac{q}{\beta(t)}\right)\leq \Psi(z)+\Psi^*\left(\frac{q}{\beta(t)}\right)=\Psi^*\left(\frac{q}{\beta(t)}\right),$$
 Condition (\ref{eq.Attouch-Czarnecki}) implies Condition (\ref{eq.Fitzpatrick:A+betaB}), which is in turn equivalent to (\ref{eq.Brezis-Haraux:A+betaB}). Hence
 all the assumptions of Theorem~\ref{co.weak_ergo_cv:A+betaB} are fullfilled.

\subsubsection{Proof of Theorem~\ref{co.weak_ergo_cv:A+betaB}}
Let us start with the following preliminary result.
\begin{lmm}\label{lm.fitz_sum}
Let $A$, $B:H \tto H$ be two monotone operators. Then the following properties hold
\begin{itemize}
\item[$(i)$] For every $(z,p)\in H\times H$,
$$G_{A+B}(z,p)\leq \inf_{q\in H}G_A(z,q)+G_B(z,p-q).$$
\item[$(ii)$] For every $(z,p)\in H\times H$ and every $\lambda>0$, \, $G_{\lambda A}(z,p)=\lambda\, G_A(z,p/\lambda)$.
\item[$(iii)$] For every $z\in \overline{\dom A}$ and $p\in N_{\overline{\dom A}}(z)$, \, $G_A(z,p)\leq G_A(z,0)$.
\end{itemize}
\end{lmm}
\begin{proof} $(i)$  Given $(z,p)\in H\times H$, the following inequality holds true
$$F_{A+B}(z,p)\leq \inf_{q\in H}\left\lbrace F_A(z,q)+F_B(z,p-q)\right\rbrace ,$$
see for example \cite[Proposition 4.2]{BauLarSen}. By subtracting $\langle z,p\rangle$ to each member, we immediately find the announced inequality.\\
$(ii)$ Let $(z,p)\in H\times H$ and $\lambda>0$. From the definition of $G_{\lambda A}(z,p)$, we have
\begin{eqnarray*}
G_{\lambda A}(z,p)&=&\sup_{(y,q)\in\gph(\lambda A)}\langle z-y,q-p\rangle\\
&=&\lambda \sup_{(y,q')\in\gph A}\langle z-y,q'-p/\lambda\rangle=\lambda \,G_A(z,p/\lambda).
\end{eqnarray*}
$(iii)$ Fix $z\in \overline{\dom A}$ and $p\in N_{\overline{\dom A}}(z)$. For every $(y,q)\in \gph A$, we have
\begin{eqnarray*}
\langle z-y,q-p\rangle&=&\langle z-y,q\rangle+ \langle y-z,p\rangle\\
&\leq&\langle z-y,q\rangle\quad \mbox{ since } p\in N_{\overline{\dom A}}(z) \mbox{ and } y\in \dom A\\
&\leq& G_A(z,0).
\end{eqnarray*}
Taking the supremum over $(y,q)\in \gph A$, we deduce that $G_A(z,p)\leq G_A(z,0)$.
\end{proof}
Let us now come back to the proof of Theorem~\ref{co.weak_ergo_cv:A+betaB}.
The main point consists in checking that the assumption (\ref{eq.cond_Brezis-Haraux}) of Theorem \ref{th.ergodic_weak_cv} is verified with
 $A_t=A+\beta(t)\, B$ and $A_\infty=A+N_C$. Let $(z,p)\in \gph(A+N_C)$. Since $p\in Az+N_C(z)$, there exists $q\in N_C(z)$ such that $p-q\in Az$.
 Observe that
 $$\begin{array}{llll}
 G_{A+\beta(t)B}(z,p)&\leq&G_{A}(z,p-q)+G_{\beta(t)B}(z,q)&\mbox{ in view of Lemma \ref{lm.fitz_sum}$(i)$,}\\
 &=&G_{\beta(t)B}(z,q) & \mbox{ since } (z,p-q)\in \gph A,\\
 &=&\beta(t) \,G_B(z,q/\beta(t))& \mbox{ in view of Lemma \ref{lm.fitz_sum}$(ii)$.}
 \end{array}$$
 The assumption $\int_0^{+\infty} \beta(t)\,G_B\left(z,{q}/{\beta(t)}\right)\, dt<+\infty$ then implies that 
 $$\int_0^{+\infty} G_{A+\beta(t) B}(z,p)\, dt<~+\infty.$$
  It suffices now to apply Theorem \ref{th.ergodic_weak_cv}.
 
 \subsection{Coupled operators with multiscale aspects: $A_t=A+\eps(t) B$ with $\eps(t) \to~0$}
 By reversing the roles of the operators $A$ and $B$ and by using a suitable time rescaling, we obtain the following consequence of Theorem \ref{co.weak_ergo_cv:A+betaB}.

\begin{crllr}\label{co.weak_ergo_cv:A+betaB:beta_to_0}
 Let $A$, $B: H\tto H$ be two maximal monotone operators such that the sets $D=A^{-1}(0)$ and $(B+N_D)^{-1}(0)$ are nonempty. Assume that
 the operator $B+N_D$ is maximal monotone. Given a map $\eps:\R_+\to \R_+^*$, assume that the operator $A+\eps (t) B$ is maximal monotone for every $t\geq 0$. Suppose additionally that $\int_0^{+\infty}\eps(t)\, dt=+\infty$ and that
 \begin{equation}\label{eq.Brezis-Haraux:A+betaB:beta_to_0} 
 \forall z\in D, \quad \forall q\in N_D(z), \qquad \int_0^{+\infty} G_A\left(z,\eps(t)\, q\right)\, dt<+\infty.
 \end{equation}
 Then for every strong global solution $x(.)$ of 
 \begin{equation}\label{MAMI_eps}\tag{MAMI$\eps$} 
 \dot x(t)+A(x(t))+\eps(t)\, B(x(t))\ni 0,
 \end{equation}
 there exists $x_\infty\in (B+N_D)^{-1}(0)$ such that 
 $$\frac{1}{t} \int_0^t x(s)\, ds\rightharpoonup~x_\infty  \mbox{ weakly in }  H, \mbox{ as } t\to +\infty.$$
 \end{crllr}
 \begin{proof}  It is done by a time rescaling, following~\cite{AttCza}.  Let us rewrite the dynamical system \eqref{MAMI_eps} as
 $$\frac{1}{\eps(t)}\dot x(t)+B(x(t))+\frac{1}{\eps(t)} A(x(t))\ni 0.$$
 Then use the time rescaling $s=\sigma(t)=\int_0^t \eps(u)\, du$. Define $y(.)$ and $\alpha(.)$ by $y(s)= x(\sigma^{-1}(s))$ and $\alpha(s)=1/\eps(\sigma^{-1}(s))$. We then have $\dot y(s)=\dot x(t)/\eps(t)$, so that $y(.)$ satisfies the following differential inclusion
 \begin{equation*}
 \dot y(s)+B(y(s))+\alpha(s)A(y(s))\ni 0.
 \end{equation*}
In terms of the variable $s$, condition (\ref{eq.Brezis-Haraux:A+betaB:beta_to_0}) can be translated as  
$$\forall z\in D, \quad \forall q\in N_D(z), \qquad \int_0^{+\infty} \alpha(s)\,G_A\left(z,\frac{q}{\alpha(s)}\right)\, ds<+\infty.$$
The assumptions of Theorem \ref{co.weak_ergo_cv:A+betaB} are satisfied, after reversing the roles of the operators $A$ and $B$. The conclusion follows immediately. 
 \end{proof}
Condition $\int_0^{+\infty}\eps (t)\, dt=+\infty$ expresses that $\eps(t)$ does not tend too fast toward zero as $t\to +\infty$. On the other hand, condition 
(\ref{eq.Brezis-Haraux:A+betaB:beta_to_0}) prevents the parameter $\eps(t)$ from converging  very slowly toward zero. Hence the conditions in Corollary \ref{co.weak_ergo_cv:A+betaB:beta_to_0} imply a moderately slow convergence $\eps(t)\to 0$ as $t\to +\infty$. Let us now analyze the case $\int_0^{+\infty}\eps (t)\, dt<+\infty$ corresponding to a fast decaying parameter.

 \begin{crllr}\label{co.weak_ergo_cv:A+epsB}
 Let $A$, $B: H\tto H$ be two maximal monotone operators such that $A+N_{\overline{\dom B}}$ is maximal monotone and $(A+N_{\overline{\dom B}})^{-1}(0)\neq\emptyset$. Given a map $\eps:\R_+\to \R_+$, assume that the operator $A+\eps (t) B$ is maximal monotone for every $t\geq 0$. Suppose additionally that $\int_0^{+\infty} \eps(t)\, dt<+\infty$ and that
 $G_B(z,0)<+\infty$ for every $z\in \dom A \cap\overline{\dom B}$.
 Then for every strong global solution $x(.)$ of 
 \begin{equation}\tag{\ref{MAMI_eps}}
 \dot x(t)+A(x(t))+\eps(t)\, B(x(t))\ni 0,
 \end{equation}

 there exists $x_\infty\in (A+N_{\overline{\dom B}})^{-1}(0)$ such that $\frac{1}{t} \int_0^t x(s)\, ds\rightharpoonup~x_\infty$ weakly in $H$, as $t\to +\infty$.
 \end{crllr}
 \begin{proof} The main point consists in checking that the assumption (\ref{eq.cond_Brezis-Haraux}) of Theorem \ref{th.ergodic_weak_cv} is verified with
 $A_t=A+\eps(t)\, B$ and $A_\infty=A+N_{\overline{\dom B}}$. Let $(z,p)\in \gph(A+~N_{\overline{\dom B}})$. Since $p\in Az+N_{\overline{\dom B}}(z)$, there exists $q\in N_{\overline{\dom B}}(z)$ such that $p-q\in Az$.
 By arguing as in the proof of Theorem~\ref{co.weak_ergo_cv:A+betaB}, we find
 $$G_{A+\eps(t)B}(z,p)\leq \eps(t) \,G_B(z,q/\eps(t)).$$
 Recalling that $q\in N_{\overline{\dom B}}(z)$, we deduce from Lemma \ref{lm.fitz_sum}$(iii)$ that $G_B(z,q/\eps(t))\leq G_B(z,0)$ and hence
 $$G_{A+\eps(t)B}(z,p)\leq \eps(t) \,G_B(z,0).$$
 Since $\int_0^{+\infty} \eps(t)\, dt<+\infty$ and $G_B\left(z,0\right)<+\infty$ by assumption, this implies that $\int_0^{+\infty} G_{A+\eps(t)\, B}(z,p)\, dt<~+\infty$. It suffices now to apply Theorem \ref{th.ergodic_weak_cv}.
 \end{proof}
 \begin{rmrk} Assume that $B=\partial \Psi$ for a lower semicontinuous convex function $\Psi:~H\to \rinf$. 
 Recalling that
 $$G_{\partial \Psi}\left(z,0\right)=F_{\partial \Psi}\left(z,0\right)\leq \Psi(z)+\Psi^*(0)=\Psi(z)-\inf_H\Psi,$$
 we deduce that $G_{\partial \Psi}\left(z,0\right)<+\infty$ if $z\in \dom \Psi$ and $\inf_H \Psi>-\infty$.
 \end{rmrk}
 
\section{Nonautonomous subgradient inclusion} \label{se.NAGI}
Let us consider the following nonautonomous subgradient inclusion
\begin{equation}\label{eq.NAGI} \tag{NAGI}
\dot x(t)+\partial \phi_t(x(t))\ni 0, \qquad t\geq 0,
\end{equation}
where for every $t\geq 0$, $\phi_t:H\to \rinf$ is a closed convex proper function.
As in Section \ref{se.NAMI}, a map $x:[0,+\infty[\to H$ is said to be a strong global solution of \eqref{eq.NAGI} if it is absolutely continuous on any bounded interval $[0,T]$, and if \eqref{eq.NAGI} holds for almost every $t>0$. Equation \eqref{eq.NAGI} is a particular case of \eqref{eq.NAMI}, since the operator $A_t=\partial \phi_t$ is maximal monotone for every $t\geq 0$. In the framework of subdifferential operators, we can make precise the convergence results of Section \ref{se.NAMI}, and show the convergence  (instead of the ergodic convergence) of the trajectories.
 
\smallskip
In the autonomous case, $\phi_t\equiv \phi$ for every $t\geq 0$, and \eqref{eq.NAGI} reduces to the steepest descent system
\begin{equation}\label{SD} \tag{SD}
\dot x(t)+\partial \phi(x(t))\ni 0, \qquad t\geq 0.
\end{equation}
Bruck~\cite[Theorem 4]{Bru} gives the weak convergence of the trajectories of \eqref{SD}, when $\argmin \phi\ne \emptyset$.  It can be derived directly from the Baillon-Br\'ezis theorem \cite{BaiBre}. The proof relies on a global estimate of the time derivative, see \cite[Theorem~5]{Bre78}, by using the equality
$$x(t)-\frac{1}{t} \int_0^t x(s)\, ds=\frac{1}{t} \int_0^t \dot{x}(s)\,s\, ds.$$
If one obtains the same estimate $\lim_{t\to +\infty}t\dot{x}(t)=0$ in the present case, the weak convergence of the trajectories of \eqref{eq.NAGI} is a direct consequence of the weak ergodic convergence of the trajectories of \eqref{eq.NAMI}. However, the extension of the energetical argument to the nonautonomous case, leading to the estimate, remains an open question in our general setting. So we provide specific results and proofs in the subgradient case.
%
\subsection{Case of a nonincreasing family $(\phi_t)_{t\geq 0}$: energetical approach}
In this subsection, we assume a monotonicity property on the filtered family $(\phi_t)_{t\geq 0}$. This allows us to use energetical arguments in order to derive convergence of the trajectories of \eqref{eq.NAGI}.
\begin{thrm}\label{th.nonincreasing}
Let $\{\phi_t; \, t\geq 0\}$ be a family of closed convex proper functions from $H$ to $\rinf$. Assume that $\phi_t\leq \phi_s$ for every $s$, $t\geq 0$ such that $s\leq t$. Let us set $\phi_\infty={\rm cl}(\inf_{t\geq 0}\phi_t)$.
Let $x(.)$ be a strong global solution of \eqref{eq.NAGI} such that the function $t\mapsto \phi_t(x(t))$ is locally absolutely continuous. Then we have 
\begin{enumerate}
\item[$(i)$] The function $t\mapsto \phi_t(x(t))$ is nonincreasing, and $\displaystyle{\lim_{t\to +\infty} \phi_t(x(t))=\inf_H \phi_\infty}$.
\end{enumerate}
Additionally assume that $\displaystyle{\inf_H \phi_\infty >-\infty}$. Then
\begin{enumerate}
\item[$(ii)$] $\displaystyle{\int_0^{+\infty}\|\dot x(t)\|^2\, dt<+\infty}$.
\end{enumerate}
Assume moreover that $S_\infty=\argmin \phi_\infty\neq \emptyset$, and that
\begin{mycondition}\label{eq.cond_L1}
\forall z\in S_\infty, \quad \int_0^{+\infty}G_{\partial \phi_t}(z,0)\, dt<+\infty.
\end{mycondition}
Then 
\begin{enumerate}
\item[$(iii)$] there exists $x_\infty\in S_\infty$ such that  $w- \lim_{t\to +\infty}x(t)= x_\infty$.
\end{enumerate}
\end{thrm}

\begin{proof} $(i)$ Let $t>0$ be such that the derivatives $\dot x(t)$ and $\frac{d}{dt}\phi_t(x(t))$ exist at $t$ and such that the inclusion $-\dot x(t)\in \partial \phi_t(x(t))$ holds true. The subdifferential inequality yields for every $\tau\in ]0,t[$
$$\phi_t(x(t-\tau))\geq \phi_t(x(t))+\langle -\dot x(t), x(t-\tau)-x(t)\rangle.$$
Recalling that the family $\{\phi_t; \, t\geq~0\}$ is nonincreasing, we have $\phi_{t-\tau}(x(t-\tau))\geq \phi_t(x(t-\tau))$, thus implying that
$$\phi_{t-\tau}(x(t-\tau))- \phi_t(x(t))\geq \langle -\dot x(t), x(t-\tau)-x(t)\rangle.$$
Dividing by $\tau$ and taking the limit as $\tau \to 0$, we find 
\begin{equation}\label{eq.ener_decay}
-\frac{d}{dt}\phi_t(x(t))\geq \|\dot x(t)\|^2\geq 0.
\end{equation}
Since this is true for almost every $t>0$, the map $t\mapsto \phi_t(x(t))$ is nonincreasing, and hence converges toward some $l\in \R\cup \{-\infty\}$. Using that $\phi_t\geq \phi_\infty$ for every $t\geq 0$, we obtain
\begin{equation}\label{eq. mino_l}
l=\lim_{t\to +\infty}\phi_t(x(t))\geq \inf_H \phi_\infty.
\end{equation}
Let us now fix $z\in H$, and define the auxiliary function $h:\R_+\to \R_+$ by $h(t)=\demi \|x(t)-z\|^2$. By differentiating, and using the subdifferential inequality, we find for almost every $t\geq 0$
\begin{eqnarray}
\dot h(t)&=&\langle x(t)-z, \dot x(t)\rangle\label{eq.der_h}\\
&\leq&\phi_t(z)-\phi_t(x(t)).\nonumber
\end{eqnarray}
Integrating this inequality, we get
$$\int_0^t\left[ \phi_s(z)-\phi_s(x(s))\right]\, ds \geq h(t)-h(0)\geq -h(0).$$
We immediately deduce that
$\lim_{s\to +\infty}\phi_s(z) \geq \lim_{s\to +\infty} \phi_s(x(s))=l.$ Since this is true for every $z\in H$, the function $\inf_{s\geq 0}\phi_s= \lim_{s\to +\infty}\phi_s$ is minorized by $l$. It ensues that the function  $\phi_\infty={\rm cl}\left(\inf_{s\geq 0}\phi_s\right)$
is also minorized by $l$. In view of (\ref{eq. mino_l}), we conclude that $l=\inf_H \phi_\infty$.\\
$(ii)$ Integrating the first inequality of (\ref{eq.ener_decay}), we find for every $t\geq 0$
$$\int_0^t\|\dot x(s)\|^2\, ds\leq \phi_0(x(0))-\phi_t(x(t)).$$
Taking the limit as $t\to +\infty$, we deduce from $(i)$ that
\begin{eqnarray*}
\int_0^{+\infty}\|\dot x(s)\|^2\, ds&\leq& \phi_0(x(0))-\inf_H \phi_\infty\\
&<&+\infty\quad \mbox{ since } \inf_H \phi_\infty>-\infty \mbox{ by assumption.}
\end{eqnarray*}
$(iii)$ The proof of the weak convergence $x(t)\rightharpoonup x_\infty$ is based on the Opial lemma. Fix $z\in S_\infty$ and consider the function $h$ defined above by $h(t)=\demi \|x(t)-z\|^2$. Coming back to equality (\ref{eq.der_h}) and recalling that $-\dot x(t)\in \partial \phi_t(x(t))$ for almost every $t\geq 0$, we find 
 $$\dot h(t)\leq G_{\partial \phi_t}(z,0)\quad \mbox{a.e. on } \R_+,$$
where $G_{\partial \phi_t}$ is the Br\'ezis-Haraux function associated to the operator $\partial \phi_t$. It follows from this inequality and assumption (\ref{eq.cond_L1}) that  $\dot h_+\in L^1(0,+\infty)$. From a classical lemma, this implies that $\lim_{t\to +\infty} h(t)$ exists in $\R$. It suffices now to prove that every 
 sequential weak cluster point of $x(.)$ belongs to $S_\infty$. Let $x_\infty\in H$ and let $t_n\to +\infty$ a sequence such that $x(t_n)\rightharpoonup x_\infty$ as $n\to +\infty$. Since the family $(\phi_t)_{t\geq 0}$ is nonincreasing, it Mosco converges toward $\phi_\infty={\rm cl}\left(\inf_{s\geq 0}\phi_s\right)$, see \cite[Theorem 3.20]{Att}. It ensues that
\begin{eqnarray*}
\phi_\infty(x_\infty)&\leq &\liminf_{n\to +\infty} \phi_{t_n}(x(t_n))\\
&=&\liminf_{t\to +\infty} \phi_{t}(x(t))=\min_H \phi_\infty \quad \mbox{ in view of $(i)$.}
\end{eqnarray*}
We conclude that $x_\infty\in S_\infty$. It suffices then to apply the Opial lemma.
\end{proof}

\begin{rmrk}\label{rk.L1_vs_L1bis}
Condition (\ref{eq.cond_L1}) is nothing else as condition (\ref{eq.cond_Brezis-Haraux}) applied with $A_t=\partial \phi_t$ and $p=0$. Recalling that
$$G_{\partial \phi_t}(z,0)=F_{\partial \phi_t}(z,0)\leq \phi_t(z)+\phi_t^*(0)=\phi_t(z)-\inf_H\phi_t,$$
we deduce that assumption (\ref{eq.cond_L1}) is implied by
\begin{mycondition}\label{eq.cond_L1_bis}
\forall z\in S_\infty, \quad \int_0^{+\infty}\left[\phi_t(z)-\inf_H\phi_t\right]\, dt<+\infty.
\end{mycondition}
\end{rmrk}

\begin{rmrk}
 Assumptions (\ref{eq.cond_L1}) and (\ref{eq.cond_L1_bis}) seem to be new in the study of the asymptotic behavior of the dynamical system \eqref{eq.NAGI}. Furuya, Miyashiba \& Kenmochi obtained the weak convergence of the trajectories
 of \eqref{eq.NAGI} under an alternative condition, see \cite[Theorem 2]{FurMiyKen}. Their condition also requires some quantity to be summable, but it differs significantly from (\ref{eq.cond_L1}) and (\ref{eq.cond_L1_bis}).
In the framework of the diagonal proximal point method, Lemaire used a discrete anologue of (\ref{eq.cond_L1_bis}) to derive the weak convergence of the iterates, see \cite[Section 4]{Lem}.
\end{rmrk}
\begin{crllr}
Let $\Psi$, $\Phi: H\to \rinf$ be closed convex functions such that $\dom \Psi \cap \dom \Phi\neq \emptyset$. Assume that $\Psi$ is nonnegative. Let $\eps: \R_+\to \R_+$
be a nonincreasing map such that $\lim_{t\to +\infty}\eps(t)=0$.
Let $x(.)$ be a strong global solution of 
\begin{equation}\tag{\ref{eq.viscosity}}
\dot x(t)+\partial \Phi(x(t))+\eps(t)\, \partial \Psi(x(t))\ni 0,
\end{equation}
such that the function $t\mapsto \Phi(x(t))+\eps(t) \Psi(x(t))$ is locally absolutely continuous.
Then we have 
\begin{enumerate}
\item[$(i)$] The function $t\mapsto \Phi(x(t))+\eps(t) \Psi(x(t))$ is nonincreasing and tends toward $\inf_H(\Phi+\delta_{\dom \Psi})$ as $t\to +\infty$.
\end{enumerate}
 Additionally assume that $\inf_H(\Phi+\delta_{\dom \Psi})>-\infty$, then
\begin{enumerate}
\item[$(ii)$]  $\int_0^{+\infty}\|\dot x(t)\|^2\, dt<+\infty$.
\end{enumerate}
Assume moreover that the set $S_\infty=\argmin({\rm cl}(\Phi+\delta_{\dom \Psi}))$ is not empty and included in $\dom \Psi$.
If $\int_0^{+\infty}\eps(t)\, dt<+\infty$, then
\begin{enumerate}\item[$(iii)$]
there exists $x_\infty\in S_\infty$ such that $w- \lim_{t\to +\infty}x(t)= x_\infty$. 
\end{enumerate}
\end{crllr}

\begin{proof} Let us check that the assumptions of Theorem \ref{th.nonincreasing} are satisfied for $\phi_t=\Phi +\eps(t)\Psi$. Since the map $\eps$ converges
nonincreasingly toward $0$, and since $\Psi$ is nonnegative, the function $ t \mapsto \phi_t$ is nonincreasing, and tends toward  $\Phi+\delta_{\dom \Psi}$ as $t\to~+\infty$.
We denote by $\phi_\infty={\rm cl}(\Phi+\delta_{\dom \Psi})$ the lower semicontinuous regularization of the function $\Phi+\delta_{\dom \Psi}$. Observing that
$\inf_H \phi_\infty=\inf_H(\Phi+\delta_{\dom \Psi})$, item $(i)$ (resp. $(ii)$) is a direct consequence of Theorem \ref{th.nonincreasing}$(i)$ (resp. $(ii)$). Let us now check that condition  (\ref{eq.cond_L1_bis}) is satisfied, thus implying the weaker condition (\ref{eq.cond_L1}). Given $z\in S_\infty$, we have
\begin{equation}\label{eq.majo_phi_t}
\phi_t(z)=\Phi(z)+\eps(t)\Psi(z)\leq \phi_\infty(z)+\eps(t)\Psi(z)=\min_H \phi_\infty +\eps(t)\Psi(z),
\end{equation}
the inequality $\Phi\leq \phi_\infty$ follows from the inequality $\delta_{\dom \Psi}\geq 0$ and the closedness of $\Phi$. On the other hand, we have 
$\phi_\infty\leq\phi_t $, and hence $\min_H \phi_\infty\leq \inf_H \phi_t $ for every $t\geq 0$. In view of (\ref{eq.majo_phi_t}), we deduce that
$$\phi_t(z)-\inf_H \phi_t\leq \eps(t)\Psi(z).$$
Since $S_\infty\subset \dom \Psi$, we have $\Psi(z)<+\infty$, and condition (\ref{eq.cond_L1_bis}) is then an immediate consequence of the assumption $\int_0^{+\infty}\eps(t)\, dt<+\infty$. Item $(iii)$ then follows directly from Theorem \ref{th.nonincreasing}$(iii)$.
\end{proof}
\subsection{A general result of convergence relying on the study of the distance to the optimal set $S_\infty$}
As in the previous subsection, $x(\cdot)$ denotes a strong global solution of the evolution inclusion \eqref{eq.NAGI}. We now study the distance 
of the solution $x(t)$ to the optimal set $S_\infty$, and we show that it vanishes as $t\to +\infty$. This is in fact an extension of a result due to Baillon-Cominetti \cite{BaiCom} in a finite dimensional framework. To obtain such an extension in a general Hilbert space, one has to assume some inf-compactness property on the functions $\phi_t$.
Let us recall that a function $f:H\to \R\cup\{-\infty,+\infty\}$ is said to be inf-compact if, for every $l\in \R$ the lower level set
$\{x\in H: \, f(x)\leq l\}$ is relatively compact in  $H$.
A weaker notion consists in requiring that the function $f+\delta_{\overline{B}(0,R)}$ is inf-compact\footnote{We use here the convention $(-\infty)+(+\infty)=+\infty$.} for every $R>0$. Here $\overline{B}(0,R)$ denotes the closed ball of radius $R$ centered at $0$. This condition amounts to assuming that for every $R>0$ and $l\in \R$ the lower level set
\begin{equation}\label{eq.inf-compact}
\{x\in H: \ \|x\|\leq R, \, f(x)\leq l\}\, \mbox{ is relatively compact in } H.
\end{equation}
If $H$ is finite-dimensional, the ball $\{x\in H: \ \|x\|\leq R \}$ is compact, and the inf-compactness property above is satisfied for every function $f:H\to \R\cup\{-\infty,+\infty\}$.

\begin{thrm}\label{th.NAGI}
Let $\{\phi_t; \, t\geq 0\}$ be a family of closed convex functions from $H$ to $\rinf$. Assume that
\footnote{For the convenience of the reader, and coherence with the literature, we keep the name of the assumptions $(\rm H 1)$-$(\rm H 2)$-$(\rm H 3)$ as in \cite{BaiCom}.
}
\begin{itemize}
\item[$(\rm H 1)$] There exists a closed proper convex function $\phi_\infty:H\to \rinf$ such that the set $S_\infty=\argmin \phi_\infty$ is nonempty and bounded.
\item[$(\rm H 2)$] $\phi_\infty(x_\infty)\leq \liminf_{k\to +\infty}\phi_{t_k}(x_k)$ for all convergent sequences $x_k\to x_\infty$ and $t_k\to +\infty$.
\item[$(\rm H 3)$] $\lim_{t\to +\infty} v_\infty (t)=\min_H \phi_\infty$, where $v_\infty(t)=\sup_{z\in S_\infty} \phi_t(z)$.
\item[$(\rm H 4)$] For $t$ large enough, all functions $\phi_t$ are uniformly minorized by a function $f:H\to\R\cup\{-\infty,+\infty\}$ satisfying\footnote{If $H=\R^n$, assumption $(\rm H 4)$ is automatically satisfied (take $f\equiv -\infty$).} the inf-compactness property (\ref{eq.inf-compact}).
\end{itemize}
Let $x(.)$ be a strong global solution of \eqref{eq.NAGI}. Then we have 
\begin{itemize}
\item[$(i)$] $\lim_{t\to +\infty} d(x(t), S_\infty)=0$.
\item[$(ii)$] If we assume moreover that
\begin{equation}\tag{\ref{eq.cond_L1}}
\forall z\in S_\infty, \quad \int_0^{+\infty}G_{\partial \phi_t}(z,0)\, dt<+\infty,
\end{equation}
then there exists $x_\infty\in S_\infty$ such that $x(t)\rightharpoonup x_\infty$ weakly in $H$ as $t\to +\infty$.
\end{itemize}
\end{thrm}

\noindent Recall that Assumption \eqref{eq.cond_L1} is satisfied under the  stronger condition 
\begin{equation}\tag{\ref{eq.cond_L1_bis}}
\forall z\in S_\infty, \quad \int_0^{+\infty}\left[\phi_t(z)-\inf_H\phi_t\right]\, dt<+\infty,
\end{equation}
see Remark \ref{rk.L1_vs_L1bis}.

\begin{proof}
$(i)$ In a finite dimensional space, Baillon \& Cominetti proved that 
$$\lim_{t\to +\infty} d(x(t), S_\infty)=~0$$
 under $(\rm H 1)$-$(\rm H 2)$-$(\rm H 3)$, see \cite[Theorem 2.1]{BaiCom}. An immediate adaptation of their arguments shows that this property still holds true
in a Hilbert space, under the additional assumption $(\rm H 4)$. \\
$(ii)$ The proof of the weak convergence $x(t)\rightharpoonup x_\infty$ is based on the Opial lemma. To show that $\lim_{t\to +\infty} \|x(t)-z\|$ exists for every $z\in S_\infty$, we use the map $h$ defined by $h(t)=\demi \|x(t)-z\|^2$ and we proceed as in the proof of Theorem \ref{th.nonincreasing}$(iii)$.
 The second point consists in proving that every weak limit point of $x(.)$ belongs to $S_\infty$. In fact, this is an immediate consequence of $(i)$ and of the weak lower semicontinuity of the convex continuous function $d(.,S_\infty)$.
\end{proof}

\subsection{Coupled gradients with multiscale aspects}
Let us now consider the case $\phi_t=\Phi+\beta(t)\Psi$, where the functions $\Phi$, $\Psi : H \to \R\cup\{+\infty\}$ are closed convex and the parameter $\beta(t)$ tends to $+\infty$.
The trajectories of the multiscale gradient dynamics\,
\begin{equation}\tag{\ref{MAG}} 
 \dot x(t)+\partial \Phi(x(t))+\beta(t)\, \partial \Psi(x(t))\ni 0
 \end{equation}
 tend to minimize the function $\Phi$ over the set $C=\argmin \Psi$. 
If the parameter $\beta(t)$ tends {\it rather fast} to $+\infty$, then any trajectory weakly converges to a point of $\argmin_C \Phi$. This is the subject of the next statement, for which we define, following \cite{Cab}, the map $\omega:\R_+\to \R\cup\{-\infty\}$:  for every $\eps\geq 0$,
\begin{equation}\label{eq.def_omega}
\omega(\eps)=\inf_H(\Psi+\eps \Phi).
\end{equation}
Corollary \ref{co.viscosity} below shows that the map $\omega$ plays a crucial role in the asymptotic study of the dynamical system (\ref{MAG}). A detailed study of the map $\omega$ will be carried out in Section \ref{se.omega}.

\begin{crllr}\label{co.viscosity} Assume
\begin{itemize}
\item[$(\cH_\Psi)$ \,]  $\Psi : H \to \R\cup\{+\infty\}$ is a closed convex proper function such that  $\inf_H \Psi =0$, and $C=~\argmin \Psi\neq~\emptyset$.
\item[$(\cH_\Phi)$ \,]  $\Phi : H \to \R\cup\{+\infty\}$ is a closed convex proper function such that  $\inf_C \Phi=0$, and $S=~\argmin_C \Phi\neq~\emptyset$.
\end{itemize}
Assume that the set $\argmin_C \Phi$ is bounded, and that the function $\Psi +\Phi$ satisfies the inf-compactness property (\ref{eq.inf-compact}). Let $\beta:\R_+\to \R_+$ be a map\footnote{Note that we do not assume any monotonicity property for the map  $\beta$.} such that $\lim_{t\to+\infty}\beta(t)=~+\infty$. 
Let $x(.)$ be a strong global solution of \eqref{MAG}.
Then we have
\begin{itemize}
\item[$(i)$] $\lim_{t\to +\infty} d(x(t), \argmin_C\Phi)=0$. In particular, if the set $\argmin_C\Phi$ is a singleton $\{\xbar\}$ for some $\xbar\in H$, then $x(t)\to \xbar$ strongly in $H$ as $t\to +\infty$.
\end{itemize}
Additionally assume that 
\begin{mycondition}\label{eq.cond_beta}
\int_0^{+\infty}\beta(t)\left|\omega (1/\beta(t))\right|\, dt<+\infty,
\end{mycondition}
Then 
\begin{itemize}
\item[$(ii)$] there exists $x_\infty\in \argmin_C \Phi$ such that  $w- \lim_{t\to +\infty}x(t)= x_\infty$.
\end{itemize}
\end{crllr}
\begin{proof} Let us check that the hypotheses of Theorem \ref{th.NAGI} are satisfied for $\phi_t=\Phi+\beta(t)\Psi$. Assumption $(\rm H 1)$ is satisfied with $\phi_\infty=\Phi+\delta_C$ and $S_\infty=\argmin_C\Phi$. Now let $(x_k)\subset H$ and $(t_k)\subset \R_+$ be such that $x_k\to x_\infty$ and $t_k\to +\infty$ as $k\to +\infty$.
Let us fix $m>0$. Since $\lim_{k\to +\infty} \beta(t_k)=+\infty$, we have $\beta(t_k)\geq m$ for $k$ large enough and hence
$$\liminf_{k\to +\infty}(\Phi(x_k)+\beta(t_k)\Psi(x_k))\geq \liminf_{k\to +\infty}(\Phi(x_k)+m\Psi(x_k)).$$
Recalling that $x_k\to x_\infty$, and that the functions $\Phi$ and $\Psi$ are closed, we deduce that
$$\liminf_{k\to +\infty}(\Phi(x_k)+\beta(t_k)\Psi(x_k))\geq \Phi(x_\infty)+m\Psi(x_\infty).$$
Letting $m\to +\infty$, we infer that
$$\liminf_{k\to +\infty}(\Phi(x_k)+\beta(t_k)\Psi(x_k))\geq \Phi(x_\infty)+\delta_C(x_\infty),$$
hence $(\rm H 2)$ is fulfilled. For every $z\in S_\infty=\argmin_C\Phi$, we have $\phi_t(z)=0$,
therefore $v_\infty(t)=0$ for every $t\geq 0$, and $(\rm H 3)$ is trivially satisfied. Since $\beta(t)\to +\infty$ we have
$\Phi+\Psi\leq \Phi+\beta(t)\Psi=\phi_t$ for $t$ large enough, and assumption~$(\rm H 4)$ is satisfied with $f=\Phi+\Psi$.
Now observe that for every $t\geq 0$ and $z\in S_\infty$,
\begin{eqnarray*}
\phi_t(z)-\inf_H \phi_t&=&-\inf_H (\Phi+\beta(t)\Psi)\quad \mbox{ since } \Phi(z)=\Psi(z)=0\\
&=&-\beta(t)\, \omega(1/\beta(t))\quad \mbox{ by definition of the map } \omega\\
&=&\beta(t)\, |\omega(1/\beta(t))|\quad \mbox{ because } \omega\leq 0.
\end{eqnarray*}
In view of Condition (\ref{eq.cond_beta}), Condition \eqref{eq.cond_L1_bis} is clearly satisfied, thus implying \eqref{eq.cond_L1}.
Conclusions $(i)$-$(ii)$ then follow from Theorem \ref{th.NAGI}.
\end{proof}

\begin{rmrk}\label{rk.strong_cv} In the context of the previous theorem, one can easily show that\footnote{Equality (\ref{eq.Psi_to0}) is a basic result which requires neither Condition (\ref{eq.cond_beta}) nor the inf-compactness of $\Phi+\Psi$.}
\begin{equation}\label{eq.Psi_to0}
\lim_{t\to +\infty} \Psi(x(t))=0,
\end{equation}
see for example \cite[Lemma 3.3]{AttCza}. Hence there exists $t_0\geq 0$ such that $\Psi(x(t))\leq 1$ for every $t\geq t_0$. Since the trajectory $x(.)$ is bounded, there exists $R>0$ such that $\|x(t)\|\leq R$ for every $t\geq 0$.
If $\Psi$ satisfies the inf-compactness property (\ref{eq.inf-compact}), we deduce that the set $\{x(t), \, t\geq t_0\}$ is relatively compact for the strong topology of $H$. Recalling from Corollary \ref{co.viscosity} $(ii)$ that the
 trajectory $x(.)$ weakly converges to~$x_\infty$, we immediately deduce that it  converges strongly to $x_\infty$.
\end{rmrk}

\begin{rmrk} Assume that the function $\Psi$ satisfies the following quadratic conditioning property 
$$\Psi\geq a\,d^2(\cdot,C) \quad \mbox{ for some $a>0$}.$$ 
Under this condition, there exists $c>0$ such that $|\omega(\eps)|\leq c\, \eps^2$ for every $\eps\geq 0$, see Section \ref{se.omega}. Hence, in this case, assumption (\ref{eq.cond_beta}) is fulfilled if $\int_0^{+\infty}(1/\beta(t))\, dt <~+\infty.$
\end{rmrk}

\begin{crllr}\label{co.inter_nonvide} Under Hypotheses $(\cH_\Psi)$-$(\cH_\Phi)$, assume that the set $S=\argmin \Psi\cap \argmin \Phi$ is nonempty and bounded. Suppose that the function $\Psi +\Phi$ satisfies the inf-compactness property (\ref{eq.inf-compact}). Let $\beta:\R_+\to \R_+$ be a map  that satisfies $\lim_{t\to+\infty}\beta(t)=~+\infty$. 
Let $x(.)$ be a strong global solution of 
\begin{equation}\tag{\ref{MAG}} 
 \dot x(t)+\partial \Phi(x(t))+\beta(t)\, \partial \Psi(x(t))\ni 0.
 \end{equation}
Then there exists $x_\infty\in S$ such that $x(t)\rightharpoonup x_\infty$ weakly in $H$ as $t\to +\infty$.
\end{crllr}

\begin{proof}
If $\argmin \Psi\cap \argmin \Phi\neq\emptyset$,
the infimum in the definition of $\omega(\eps)$ is attained at every $x\in \argmin \Psi\cap \argmin \Phi$, and it equals $0$. It ensues that $\omega(\eps)=0$ for every $\eps\geq 0$. Therefore Condition\eqref{eq.cond_beta}
 of Corollary \ref{co.viscosity} is automatically satisfied. 
\end{proof}

As a consequence  of Corollary \ref{co.viscosity}, we recover the convergence result of the trajectories of \eqref{MAG} from \cite{AttCza}. 

\begin{crllr} \cite[Theorem 5.1]{AttCza}\label{co.AttCza}
Let $\Psi$, $\Phi:H\to \rinf$ be functions satisfying Hypotheses $(\cH_\Psi)$-$(\cH_\Phi)$, together with the following qualification condition
\begin{equation}\tag{QC}\label{QC}
\mbox{there exists $x_0\in C$ such that $\Phi$ is continuous at $x_0$.}
\end{equation}
 Assume that the set $S=\argmin_C \Phi$ is bounded, and that the function $\Psi +\Phi$ satisfies the inf-compactness property (\ref{eq.inf-compact}). Let $\beta:\R_+\to \R_+$ be a map such that $\lim_{t\to+\infty}\beta(t)=~+\infty$. 
Assume moreover that
\begin{equation}\tag{\ref{eq.Attouch-Czarnecki}}
 \forall p\in \ran(N_C), \quad \displaystyle \int_{0}^{+\infty} \beta (t) \left[\Psi^* \left(\frac{p}{ \beta (t)}\right) - \sigma_C \left(\frac{p}{ \beta (t)}\right)\right]dt < + \infty.
\end{equation}
Let $x(.)$ be a strong global solution of 
\begin{equation}\tag{\ref{MAG}} 
 \dot x(t)+\partial \Phi(x(t))+\beta(t)\, \partial \Psi(x(t))\ni 0.
 \end{equation}
Then there exists $x_\infty\in S$ such that $x(t)\rightharpoonup x_\infty$ weakly in $H$ as $t\to +\infty$.
\end{crllr}
\begin{proof} It relies on the study of the map $\omega$ that we carry out in Section~\ref{se.omega}. Precisely, it is a consequence of the forthcoming Proposition \ref{pr.majo} $(d)$.
\end{proof}

Let us now give an equivalent formulation of Corollary \ref{co.viscosity}, which uses an asymptotic vanishing coefficient $\eps (t)$. The statement is obtained by reversing the roles of the functions $\Phi$ and $\Psi$, and by using a suitable time rescaling, which allows to pass from $\beta(t) \to + \infty$ to $\eps(t) \to 0$, and vice versa.
\begin{crllr}\label{co.viscosity_eps}
Let $\Phi : H \to \R\cup\{+\infty\}$ be a closed convex proper function such that $C=\argmin \Phi\neq \emptyset$  and $\min_H \Phi =0$.
Let $\Psi : H \to \R\cup\{+\infty\}$ be a closed convex proper function such that $\argmin_C \Psi\neq \emptyset$ and $\min_C \Psi=0$.
Assume that the set $\argmin_C \Psi$ is bounded and that the function $\Phi +\Psi$ satisfies the inf-compactness property (\ref{eq.inf-compact}). Let $\eps:\R_+\to \R_+^*$ be a map\footnote{Note that we do not assume the map $\eps$ to be nonincreasing.} such that $\lim_{t\to+\infty}\eps(t)=0$ and $\int_0^{+\infty}\eps(t)\, dt=+\infty$. 
Let $x(.)$ be a strong global solution of 
\begin{equation}\tag{\ref{eq.viscosity}}
\dot x(t)+\partial \Phi(x(t))+\eps(t)\, \partial \Psi(x(t))\ni 0,
\end{equation}
Then we have
\begin{itemize}
\item[$(i)$] $\lim_{t\to +\infty} d(x(t), \argmin_C\Psi)=0$. In particular, if the set $\argmin_C\Psi$ is a singleton $\{\xbar\}$ for some $\xbar\in H$, then $x(t)\to \xbar$ strongly in $H$ as $t\to +\infty$.
\end{itemize}
Additionally assume that 
$$\int_0^{+\infty}\left|\omega (\eps(t))\right|\, dt<+\infty,$$
 where the map $\omega:\R_+\to \R$  is 
defined by $\omega(\eps)=\inf_H(\Phi+\eps\Psi)$.
Then 
\begin{itemize}
\item[$(ii)$] 
there exists $x_\infty\in \argmin_C \Psi$ such that  $w- \lim_{t\to +\infty}x(t)= x_\infty$.
\end{itemize}
\end{crllr}
The proof is based on a suitable time rescaling, see the proof of Corollary \ref{co.weak_ergo_cv:A+betaB:beta_to_0}. 
\begin{rmrk} Cominetti-Peypouquet-Sorin\cite{ComPeySor} pay a special attention to the following steepest descent system with vanishing Tikhonov regularization
\begin{equation}\label{eq.grad_Tikhonov}\tag{SD$_\eps$}
\dot x(t)+\partial \Phi(x(t))+\eps(t)\, x(t)\ni 0.
\end{equation}
If $\int_0^{+\infty}\eps(t)\, dt=+\infty$, it is proved in \cite{ComPeySor} that any solution $x(\cdot)$ of
\eqref{eq.grad_Tikhonov} strongly converges as $t\to +\infty$ toward the least-norm minimizer of $\Phi$. With essentially the same arguments, if the function $\Psi$ is strongly convex and if $\int_0^{+\infty}\eps(t)\, dt=+\infty$,
then any solution of \eqref{eq.viscosity}
strongly converges as $t\to +\infty$ toward the unique minimizer of $\Psi$ over the set $C=\argmin \Phi$. This convergence result can be recovered from Corollary \ref{co.viscosity_eps} $(i)$. Notice that in this framework, 
the inf-compactness property required by Corollary \ref{co.viscosity_eps} appears to be superfluous.
\end{rmrk}

\medskip

Let us now consider some other examples of nonautonomous subgradient inclusions. They illustrate the versatility of our approach, and its limits.

\subsection{Quasi-autonomous case}
Let us consider the quasi-autonomous subgradient inclusion $\dot x(t)+\partial \Phi(x(t))\ni f(t)$, where $\Phi:H\to \rinf$ is a closed convex function, and the map $f:\R_+\to H$ tends to $f_\infty\in H$ as $t\to +\infty$. This differential inclusion falls into the setting of Theorem \ref{th.NAGI}, by taking $\phi_t=\Phi-\langle f(t),.\rangle$.
We obtain the following statement.

\begin{crllr}\label{co.quasi-autonomous}
Let $f:\R_+\to H$ be a map such that $\lim_{t\to+\infty}f(t)=f_\infty\in H$.
Let $\Phi:H\to \rinf$ be a closed convex function such that $S=\argmin (\Phi-\langle f_\infty,\cdot\rangle)$ is nonempty\footnote{By writing down the optimality condition for the elements of $S$, we immediately see
that $S=\partial \Phi^{-1}(f_\infty)$. It ensues that the nonvacuity of $S$ is equivalent to the condition $f_\infty\in \ran(\partial \Phi)$.} and bounded. Suppose that the function $\Phi$ satisfies the inf-compactness property (\ref{eq.inf-compact}). 
Let $x(.)$ be a strong global solution of 
\begin{equation}\label{eq.quasi-autonomous}
x(t)+\partial \Phi(x(t))\ni f(t).
\end{equation} 
Then we have
\begin{itemize}
\item[$(i)$] $\lim_{t\to +\infty} d(x(t), S)=0$.
\item[$(ii)$]  Assume moreover that 
\begin{equation}\label{eq.cond_l1_quasi-autonomous_1}
\forall z\in S, \quad \int_0^{+\infty}G_{\partial \Phi}(z,f(t))\, dt<+\infty.
\end{equation}
 Then there exists $x_\infty\in S$ such that $x(t)\rightharpoonup x_\infty$ weakly in $H$ as $t\to +\infty$.
 \end{itemize}
\end{crllr}

\begin{rmrk}
 Assumption (\ref{eq.cond_l1_quasi-autonomous_1}) is satisfied under the following stronger condition
 \begin{equation}\label{eq.cond_l1_quasi-autonomous}
\forall z\in S, \quad \int_0^{+\infty}\left[\Phi^*(f(t))+\Phi(z)-\langle f(t), z\rangle\right]\, dt<+\infty.
\end{equation}
Indeed, it suffices to observe that
$G_{\partial \Phi}(z,f(t))\leq\Phi^*(f(t))+\Phi(z)-\langle f(t), z\rangle$.
\end{rmrk}

\begin{proof} Without loss of generality, we assume that $f_\infty=0$ and that $\min_H \Phi=0$.
Let us check that the hypotheses of Theorem \ref{th.NAGI} are satisfied for $\phi_t=\Phi-\langle f(t),.\rangle$. 
Assumption $(\rm H 1)$ is satisfied with $\phi_\infty=\Phi$. Let $(t_k)\subset \R_+$ and $(x_k)\subset H$ be sequences such that $t_k\to +\infty$ and $x_k\to x_\infty$ as $k\to +\infty$, for some $x_\infty\in H$. Observe that
\begin{eqnarray*}
\liminf_{k\to +\infty}\phi_{t_k}(x_k)&=&\liminf_{k\to +\infty}(\Phi(x_k)-\langle f(t_k),x_k\rangle)\\
&=&\liminf_{k\to +\infty}\Phi(x_k) \quad \mbox{ because $f(t_k)\to 0$ and $x_k\to x_\infty$,}\\
&\geq& \Phi(x_\infty) \quad \mbox{ since $\Phi$ is closed and $x_k\to x_\infty$.}
\end{eqnarray*}
It ensues that $(\rm H 2)$ is fulfilled. From the definition of $v_\infty(t)$, we have
\begin{eqnarray*}
v_\infty(t)=\sup_{z\in S}\phi_t(z)&=&\sup_{z\in S}(\Phi(z)-\langle f(t),z\rangle)\\
&=&\sup_{z\in S}(-\langle f(t),z\rangle),
\end{eqnarray*}
because $\Phi(z)=0$ for every $z\in S=\argmin \Phi$.
 Since the set $S$ is bounded, there exists $M>0$ such that $\|z\|\leq M$ for every $z\in S$. It follows immediately that $|v_\infty(t)|\leq M\, \|f(t)\|\to 0$ as $t\to +\infty$, which shows that $(\rm H 3)$ is satisfied. For every $x\in H$ and $t$ large enough, we have
 \begin{eqnarray*}
 \phi_t(x)=\Phi(x)-\langle f(t),x\rangle&\geq& \Phi(x)-\|f(t)\|\|x\|\\
 &\geq& \Phi(x)-\|x\| \quad \mbox{ since $f(t)\to 0$ as $t\to +\infty$.}
 \end{eqnarray*}
It suffices to check that the function $x\mapsto \Phi(x)-\|x\|$ satisfies the inf-compactness property (\ref{eq.inf-compact}). For every $R>0$ and $l\in \R$, we have    
$$\{x\in H: \, \|x\|\leq R, \, \Phi(x)-\|x\|\leq l\}\subset \{x\in H: \,\|x\|\leq R, \,\Phi(x)\leq l+R\}.$$
This last set is relatively compact by assumption, hence hypothesis $(\rm H 4)$ is satisfied. Finally observe that $G_{\partial \phi_t}(z,0)=G_{\partial \Phi}(z,f(t))$
for every $z\in S$.
In view of assumption (\ref{eq.cond_l1_quasi-autonomous_1}), Condition \eqref{eq.cond_L1} is clearly verified.
Conclusions $(i)$-$(ii)$ then follow from Theorem \ref{th.NAGI}.
\end{proof}
The next proposition gives sufficient conditions which guarantee that assumptions 
\eqref{eq.cond_l1_quasi-autonomous_1} and \eqref{eq.cond_l1_quasi-autonomous} are satisfied.
\begin{prpstn}\label{pr.quasi-autonomous}
Let $f:\R_+\to H$ be a map such that $\lim_{t\to+\infty}f(t)=f_\infty\in H$. Let $\Phi:H\to \rinf$ be a closed convex function such that the set $S=\argmin \big(\Phi-~\langle f_\infty,\cdot \rangle\big)$ is nonempty and bounded. 
The following hold true
\begin{itemize}
\item[$(i)$] If the function $\Phi-\langle f_\infty,\cdot \rangle$ is coercive and if $\int_0^{+\infty}\|f(t)-f_\infty\|\, dt<+\infty$, then Condition (\ref{eq.cond_l1_quasi-autonomous}) is satisfied.
\item[$(ii)$] Assume that 
\begin{equation}\label{eq.quadratic_conditioning}
\Phi-\langle f_\infty,\cdot \rangle- \min_H\big(\Phi-\langle f_\infty,\cdot \rangle\big)\geq a\, d^2(.,S)
\end{equation} for some $a>0$ and that
\begin{equation}\label{eq.cond_l1_l2}
\int_0^{+\infty}\left\|\Pi_F (f(t)-f_\infty)\right\|\, dt<+\infty \quad \mbox{ and } \quad
\int_0^{+\infty}\left\|\Pi_{F^\perp} (f(t)-f_\infty)\right\|^2\, dt<+\infty,
\end{equation}
where $\Pi_F$ (resp. $\Pi_{F^\perp}$) denotes the orthogonal projection on the linear space $F=\cl[\R_+(S-S)]$ (resp. $F^\perp$).
Then condition (\ref{eq.cond_l1_quasi-autonomous}) is satisfied.
\end{itemize}
\end{prpstn}
\begin{proof} Without loss of generality, we assume that $f_\infty=0$ and that $\min_H \Phi=0$.

$(i)$ Since $\Phi$ is supposed to be coercive, the conjugate $\Phi^*$ is continuous at $0$. It
ensues classically that $\Phi^*$ is Lipschitz continuous in a neighborhood of $0$. Therefore there exist $r>0$ and $L>0$ such that $|\Phi^*(x)|\leq L\|x\|$ for every
$x\in H$ satisfying $\|x\|\leq r$. Let $M>0$ be such that $\|z\|\leq M$ for every $z\in S$. Recalling that $f(t)\to 0$ as $t\to +\infty$, we deduce that for $t$ large enough
$$\Phi^*(f(t))-\langle f(t),z\rangle \leq L\|f(t)\|+\|f(t)\|\|z\|\leq (L+M)\|f(t)\|.$$
The assumption $\int_0^{+\infty}\|f(t)\|\, dt<+\infty$, and $\Phi (z)= \min_H \Phi =0$,  clearly imply that Condition (\ref{eq.cond_l1_quasi-autonomous}) is satisfied.

$(ii)$ Since $\Phi\geq a\, d^2(.,S)=a\|\,.\,\|^2\tdown\delta_S$, we have $\Phi^*\leq \frac{1}{4a}\|\,.\,\|^2+\sigma_S$. This implies that for every $z\in S$,
\begin{eqnarray}
\Phi^*(f(t))-\langle f(t),z\rangle &\leq& \frac{1}{4a}\|f(t)\|^2+\sigma_S(f(t))-\langle f(t),z\rangle\nonumber\\
&=&\frac{1}{4a}\|f(t)\|^2+\sigma_{S-z}(f(t))\nonumber\\
&\leq&\frac{1}{4a}\|f(t)\|^2+\sigma_{S-S}(f(t)).\label{eq.majo_sigma_S-S}
\end{eqnarray} 
From $f(t)=\Pi_F f(t)+\Pi_{F^\perp} f(t)$, we deduce that
\begin{eqnarray*}
\sigma_{S-S}(f(t))&\leq &\sigma_{S-S}(\Pi_F f(t))+\sigma_{S-S}(\Pi_{F^\perp} f(t))\quad \mbox{ since $\sigma_{S-S}$ is subadditive}\\
&\leq& 2\, M\left\|\Pi_F f(t)\right\| +\sigma_F\left(\Pi_{F^\perp} f(t)\right) \ \ \mbox{ since $S-S\subset 2\, M\BB$, $S-S\subset F$}\\
&=&2\, M\left\|\Pi_F f(t)\right\| \quad \mbox{ since $\sigma_F=\delta_{F^\perp}$.}
\end{eqnarray*}  
Coming back to inequality (\ref{eq.majo_sigma_S-S}), and using Pythagoras' equality 
$$\|f(t)\|^2=\|\Pi_F f(t)\|^2+\|\Pi_{F^\perp} f(t)\|^2,$$
 we infer that
$$\Phi^*(f(t))-\langle f(t),z\rangle \leq\frac{1}{4a}\|\Pi_F f(t)\|^2+\frac{1}{4a}\|\Pi_{F^\perp} f(t)\|^2+2\, M\left\|\Pi_F f(t)\right\|.$$
Since $f(t)\to 0$ as $t\to +\infty$, we have $\|\Pi_F f(t)\|^2\leq \|\Pi_F f(t)\|$ for $t$ large enough. Assumption (\ref{eq.cond_l1_l2}) then clearly implies that condition (\ref{eq.cond_l1_quasi-autonomous}) is satisfied.
\end{proof}

\begin{rmrk} 
By combining Corollary \ref{co.quasi-autonomous} and Proposition \ref{pr.quasi-autonomous} $(i)$, we derive that if $\int_0^{+\infty}\|f(t)-~f_\infty\|\, dt<~+\infty$, then any trajectory of (\ref{eq.quasi-autonomous})  converges weakly toward some point of $S=\partial \Phi^{-1}(f_\infty)$. This result can be recovered directly by using the Opial lemma and the fact that the energy function $t\mapsto \Phi(x(t))-\langle f_\infty,x(t)\rangle$ tends toward its minimum as $t\to +\infty$. The inf-compactness assumption on $\Phi$ appears to be useless, hence the result obtained as a consequence of Corollary \ref{co.quasi-autonomous} and Proposition \ref{pr.quasi-autonomous} $(i)$ is not optimal.
The original part of Proposition \ref{pr.quasi-autonomous} lies in point~$(ii)$, which brings to light that the $L^1$-type condition on the function $f-f_\infty$ may be relaxed. If we assume the quadratic conditioning property (\ref{eq.quadratic_conditioning}), Proposition~\ref{pr.quasi-autonomous} $(ii)$ shows that it is enough to require a $L^2$-type condition for the part of $f-f_\infty$ that is projected on $F^\perp$.
\end{rmrk}

\subsection{Sweeping process}
The sweeping process was originally  considered by J.J. Moreau in the study of evolution problems from unilateral  mechanics. 

Given $t \mapsto C(t)$ a time-dependent closed convex set in $H$, (the moving constraint), and $\Phi : H \to \mathbb R$ a convex differentiable function (the driving force), it consists in the study of the following differential inclusion
\begin{equation}\label{eq.sweep1}\tag{SW}
\dot x(t)+ N_{C(t)} (x(t)) + \nabla \Phi (x(t))\ni 0, \qquad t\geq 0,
\end{equation}
where $N_{C(t)}(x)$ stands for the normal cone to $C(t)$ at $x\in C(t)$.
Since, its range of applications has been extended to various domains, like economical and social sciences, control theory.
An abundant litterature has been devoted to its study, but curiously only few results concern its asymptotical behaviour.

The differential inclusion (\ref{eq.sweep1}) falls in the setting of Theorems \ref{th.nonincreasing} and \ref{th.NAGI}, by taking 
$$\phi_t = \delta_{C(t)}(\cdot) +\Phi.$$  The monotonicity assumption required by Theorem \ref{th.nonincreasing} amounts to saying that the family $\{C(t); \, t\geq~0\}$ 
is nondecreasing for the set inclusion. On the other hand, it is easy to check that assumptions (H2)-(H3) of Theorem \ref{th.NAGI} imply that the set $C(t)$  tends toward $C_\infty$ as $t\to +\infty$ in the Painlev\'e-Kuratowski sense and that $C_\infty\subset C(t)$ for $t$ large enough. These assumptions on the family $\{C(t); \, t\geq 0\}$ are clearly quite stringent, and it is better to work directly with inclusion (\ref{eq.sweep1}), without resorting to the general results mentioned above.

For simplicity, we assume in the sequel that $\Phi=0$. Most of the existence results concerning (\ref{eq.sweep1}) rely on energy estimates. Thus we take for granted that the trajectories have finite energy, i.e.,  $\int_0^{+\infty} \|\dot{x} (t)\|^2 dt < + \infty $. The result stated below is an illustration of the energetical methods.

\begin{thrm}\label{th.sweep}
Let $\{C(t); \, t\geq 0\}$ be a family of closed convex sets in $H$. Assume that $C(t)$ converges to some $ C_\infty$ in the Mosco sense and that, 
\begin{equation}\label{eq.sweep2}
\forall z\in C_{\infty}, \ \exists z(t)\to z \mbox{ such that }  z(t) \in C(t), \mbox{ and }\ \int_0^{+\infty} \| z(t) - z\|^2 dt < + \infty.
\end{equation}
Let $x(.)$ be a strong global solution of (\ref{eq.sweep1}) which has a finite energy, i.e.,
\begin{equation}\label{eq.sweep20}
 \int_0^{+\infty} \|\dot{x} (t)\|^2 dt < + \infty.
\end{equation}
 Then, there exists $x_\infty\in C_\infty$ such that $x(t)\rightharpoonup x_\infty$ weakly in $H$ as $t\to +\infty$.
\end{thrm}

\begin{proof}
Let us apply the Opial lemma to $x(.)$ and $S= C_{\infty}$.
Given $z \in C_{\infty}$, set $h_z (t)= \frac {1}{2} \| x(t) -z\|^2$.
By differentiating, we find for almost every $t\geq 0$
\begin{equation}
\dot h_z(t)=\langle x(t)-z, \dot x(t)\rangle.\label{eq.der_h1}
\end{equation}
By assumption  (\ref{eq.sweep2}) there exists
$z(t)\to z \mbox{ such that }  z(t) \in C(t), \mbox{ and }\ \int_0^{+\infty} \| z(t) - z\|^2 dt < + \infty.$
By introducing $z(t)$ in (\ref{eq.der_h1}), we obtain
$$\dot h_z(t)= \langle x(t)-z(t), \dot x(t)  \rangle+\langle z(t) -z, \dot x(t)\rangle.$$ 
Since $-\dot x(t) \in  N_{C(t)} (x(t))$ and $z(t) \in C(t)$, we have
$$\langle x(t)-z(t), \dot x(t)  \rangle= \langle z(t)-x(t), -\dot x(t)  \rangle\leq 0. 
$$
Hence
\begin{eqnarray}
\dot h_z(t) &\leq& \langle z(t) -z, \dot x(t)\rangle. \nonumber\\
&\leq& \| z(t) -z\| \|\dot x(t)\|\quad \mbox{by Cauchy-Schwarz inequality}.\label{eq.der_h6}
\end{eqnarray}
By assumptions (\ref{eq.sweep2}) and (\ref{eq.sweep20}) the second member of (\ref{eq.der_h6}) belongs to $L^1 (0,  +\infty)$.
Hence, the limit of $ h_z(t)$ exists as $t \to + \infty$.

On the other hand since $x(t) \in C(t)$, and $C(t)$ Mosco converges to $C_{\infty}$, we have that any weak cluster point of  the trajectory belongs to $C_{\infty}$.
Thus the two conditions of the Opial lemma are satisfied, which gives the weak convergence of the trajectory.

\end{proof}

\subsection{Slow case and strong attraction of the optimal path}
In this subsection, we assume that for every $t\geq 0$, there exist $\xi(t)\in H$ and $\alpha(t)>0$ such that
$$\forall x\in H, \quad \phi_t(x)\geq \phi_t(\xi(t))+\alpha(t)\, \|x-\xi(t)\|^2.$$\\
It implies that $\xi(t)$ is a strong minimum of the function $\phi_t$. 

\begin{rmrk} Fix $z\in S_\infty$. We deduce from the above condition that
$$\alpha(t)\,\|z-\xi(t)\|^2\leq v_\infty(t)-\min_H\phi_t.$$
If $\xi^*=\lim_{t\to +\infty}\xi(t) $ exists and is not equal to $z$, there
exists $m>0$ such that $\|z-\xi(t)\|\geq m$ for $t$ large enough. It ensues that
$$\alpha(t)\leq \frac{1}{m^2}(v_\infty(t)-\min_H\phi_t)\quad \mbox{ for $t$ large enough.}$$
\end{rmrk}
We assume that the function $\alpha$ is measurable and satisfies
$$\int_0^{+\infty}\alpha(t)\, dt=+\infty,$$
which corresponds to a slow decay condition.
Let us first consider the case of an optimal trajectory having a finite length.
The following result is a variant of \cite[Theorem 3.2]{AttCom}, up to a slight modification of the arguments.\footnote{A strong convexity property is required in the statement of  \cite[Theorem 3.2]{AttCom}.  The strong convexity property is relaxed and replaced here with the strong minimum property $(i)$.}
\begin{thrm}\label{th.selection_short}
Let $\{\phi_t, \, t\geq 0\}$ be a family of closed convex functions from $H$ to $\rinf$. Assume that
\begin{itemize}
\item[$(i)$] $\forall x\in H, \quad \phi_t(x)\geq \phi_t(\xi(t))+\alpha(t)\, \|x-\xi(t)\|^2$;
\item[$(ii)$] $\int_0^{+\infty}\alpha(t)\, dt=+\infty$;
\item[$(iii)$] the optimal path $\xi(.)$ is locally absolutely continuous on $\R_+$, and satisfies $\int_0^{+\infty}\|\dot \xi(t)\|\, dt<~+\infty.$
\end{itemize}
If $x(.)$ is a strong global solution of \eqref{eq.NAGI},
then $\lim_{t\to +\infty}\|x(t)-\xi(t)\|=0$, and hence $\lim_{t\to +\infty}x(t)=\xi^*$ strongly in $H$, where $\xi^*$ is
the limit of the optimal path $\xi(t)$ as $t\to +\infty$.
\end{thrm}
\begin{proof}
Consider the function $k$ defined by $k(t)=\demi \|x(t)-\xi(t)\|^2$. This function is absolutely continuous, and for almost every $t\in ]0,+\infty[$ we have
\begin{eqnarray*}
\dot k(t)&=& \langle \dot x(t)-\dot \xi(t), x(t)-\xi(t)\rangle\\
&\leq & \langle \dot x(t), x(t)-\xi(t)\rangle +\|\dot \xi(t)\| \|x(t)-\xi(t)\|.
\end{eqnarray*}
Since $-\dot x(t)\in \partial \phi_t(x(t))$, we deduce
from the subdifferential inequality
$$\dot k(t)+\phi_t(x(t))-\phi_t(\xi(t))\leq\|\dot \xi(t)\| \|x(t)-\xi(t)\|.$$
Invoking Assumption $(i)$, we get
$$\dot k(t)+\alpha(t)\, \|x(t)-\xi(t)\|^2\leq \|\dot \xi(t)\| \|x(t)-\xi(t)\|,$$
or equivalently
$$\dot k(t)+2\alpha(t)\, k(t)\leq \sqrt{2}\|\dot \xi(t)\| \sqrt{k(t)}.$$
The rest of the proof is analogous to that of \cite[Theorem 3.2]{AttCom}.
\end{proof}

Let us now consider the case of an optimal trajectory satisfying $\|\dot \xi(t)\| =o(\alpha(t))$ as $t\to +\infty$, see \cite[Theorem 3.3]{AttCom}.
\begin{thrm}\label{th.selection_long} Under the assumptions $(i)$ and $(ii)$ of Theorem~\ref{th.selection_short}, assume moreover that the optimal path $\xi(.)$ is locally absolutely continuous on $\R_+$ and that $\lim_{t\to+\infty}\|\dot \xi(t)\|/\alpha(t)=0.$
 Let $x(.)$ be a strong global solution of \eqref{eq.NAGI}.
 Then $\lim_{t\to +\infty}\|x(t)-~\xi(t)\|=~0$, therefore it converges strongly in $H$ if and only if the optimal path $\xi(t)$ has a limit as $t\to+\infty$.
\end{thrm}
For the proof of this result, the reader is referred to \cite[Theorem 3.3]{AttCom}.

\section{Infimum value associated to the viscosity minimization problem $\inf_H(\Psi+\eps \Phi)$}\label{se.omega}
\subsection{Main properties of the map $\eps\mapsto \omega(\eps)=\inf_H(\Psi+\eps \Phi)$}
As we have already pointed out, the map $\omega$ plays a crucial role in the asymptotic study of the dynamic system \eqref{MAG}.  We now make a systematic study of this function. 
Throughout this section, we assume $(\cH_\Psi)$ and $(\cH_\Phi)$, i.e.,
\begin{itemize}
\item[$(\cH_\Psi)$ \,]  $\Psi : H \to \R\cup\{+\infty\}$ is a closed convex proper function such that  $\inf_H \Psi =0$, and $C=~\argmin \Psi\neq~\emptyset$.
\item[$(\cH_\Phi)$ \,]  $\Phi : H \to \R\cup\{+\infty\}$ is a closed convex proper function such that  $\inf_C \Phi=0$, and $S=~\argmin_C \Phi\neq~\emptyset$.
\end{itemize}
 Recall the definition \eqref{eq.def_omega} of the map $\omega:\R_+\to \R\cup\{-\infty\}$:  for every $\eps\geq 0$
\begin{equation*}
\omega(\eps)=\inf_H(\Psi+\eps \Phi).
\end{equation*} 
We denote by $(\cP_\eps)$ the corresponding minimization problem 
$$\inf_{x\in H}\left\{\Psi(x)+\eps\, \Phi(x)\right\}.\leqno (\cP_\eps)$$
\begin{rmrk}
Assumption $(\cH_\Phi)$ implies that the domain of $\Phi$
intersects the set $C$ of minimizers of $\Psi$. This corresponds to a regular perturbation situation, where we can expect a simple asymptotic development for $\omega(\eps)$ as $\eps$ goes to zero, as well as the convergence of the filtered sequence of solutions of $(\cP_\eps)$ to a solution of the hierarchical minimization problem $\min_C \Phi$. That is the situation we consider. By contrast, when the domain of $\Phi$ does not intersect the set $C=\argmin \Psi$, we are faced with a singular perturbation. This is a more involved situation, that one encounters for example in  phase transition, when considering the  Van der Waals-Cahn-Hilliard viscous approximation of the Gibbs free energy. In this case, we must appeal to  $\Gamma$-convergence methods for rescaled energy functions, see \cite{Att2}, \cite[Chap. 12.5]{AttButMic},  \cite{Tor2}. 
\end{rmrk}
The following proposition gathers properties of the map $\omega$.
\begin{prpstn}\label{pr.basic}
Assume Hypotheses $(\cH_\Psi)$-$(\cH_\Phi)$. \\
(a) The map $\eps\mapsto \omega(\eps)$ is nonpositive, nonincreasing and concave on $\R_+$.\\
 Assume moreover that the function $\Psi+\Phi$ is coercive\footnote{The coercivity of $\Psi +\Phi$ implies that of $\delta_C+\Phi$, and we deduce classically that $\argmin_C \Phi=\argmin (\delta_C+\Phi)\neq\emptyset$.}. Then \\
(b)
for every $\eps\in [0,1]$, we have $\omega(\eps)>-\infty$,  and the infimum is attained in the definition of $\omega(\eps)$.\\
(c) $\lim_{\eps\to 0^+}\omega(\eps)/\eps=0$. In other words, the following asymptotic expansion holds\footnote{For simplicity, we assumed  $\min_H\Psi= \min_C\Phi=0$. The statement remains valid without any assumption on the (finite) values of $\min_H\Psi$ and $\min_C\Phi$. The asymptotic expansion (\ref{eq.asymp_expan}) 
can be found in \cite[Theorem 2.5]{Att2}.} as $\eps\to 0$
\begin{equation}\label{eq.asymp_expan}
\min_H(\Psi+\eps\Phi)=\min_H\Psi+\eps \min_C\Phi+o(\eps).
\end{equation}
\end{prpstn}
\begin{proof} (a)
Given $z\in S$, we have $$\omega(\eps)\leq \Psi (z)+\eps \Phi (z)=0,$$
hence  $\omega(\eps)\leq 0$ for every $\eps\geq 0$.
Observe that the map $\eps\mapsto \Psi(x)+\eps\Phi(x)$ is affine, hence the map $\eps\mapsto \omega(\eps)$ is concave as an infimum
of affine functions.
Since the function $\omega:\R_+\to \R\cup\{-\infty\}$ is concave, it admits a right (resp. left) derivative at every $t\geq 0$ (resp. $t>0$). In particular, we have
$$\omega'_+(0)=\lim_{\eps\to 0^+}\frac{1}{\eps}(\omega(\eps)-\omega(0))\leq 0,$$
since $\omega(0)=0$, and $\omega(\eps)\leq 0$ for every $\eps >0$. The concavity of $\omega$ implies that $\omega'_+(\eps)\leq 0$ (resp. $\omega'_-(\eps)\leq 0$) for every $\eps>0$. We deduce that the function $\omega$ is nonincreasing on $\R_+$.\\
(b)  First observe that the conclusion is immediate for $\eps=0$. 
Now assume that $\eps\in ]0,1]$. Since $\Psi(x)\geq 0$, we have
$$\Psi(x)+\eps\, \Phi(x)\geq \eps\, (\Psi(x)+\Phi(x)).$$
From the coercivity of $\Psi+\Phi$, we deduce that the lower semicontinuous convex function $x\mapsto \Psi(x)+\eps\, \Phi(x)$ is coercive. It ensues classically that the minimization problem $(\cP_\eps)$ has at least one solution,
and that $\omega(\eps)=\inf \cP_\eps>-\infty$.\\
(c) Let us argue by contradiction and assume that there exist
$\eta>0$ and a sequence $(\eps_n)$ tending toward $0$ such that $\omega(\eps_n)/\eps_n\leq -\eta$. 
From the definition of $\omega(\eps_n)$, there exists a sequence $(x_n)$ in $H$ such that
\begin{equation}\label{eq.majopsi+phi}
\forall n\in \N,\quad \Psi(x_{n}) +\, \eps_n\,\Phi(x_{n})
\leq -\frac{\eta}{2}\, \eps_n.
\end{equation}
Since $\lim_{n\to +\infty}\eps_n=0$ and $\Psi(x_n)\geq 0$, we have
$\eps_n\, \Psi(x_n)\leq \Psi(x_n)$ for $n$ large enough, say $n\geq n_0$. In view of (\ref{eq.majopsi+phi}), this implies that for every $n\geq n_0$
\begin{equation}\label{eq.majopsi+phi_bis}
\Psi(x_{n}) +\, \Phi(x_{n}) \leq -\frac{\eta}{2},
\end{equation}
or equivalently 
$$x_n\in \left[\Psi+\Phi\leq  -\frac{\eta}{2}\right].$$
Recalling that the function $\Psi+\Phi$ is coercive by assumption, we deduce that the 
sequence $(x_{n})$ is bounded in $H$. Therefore there exist $x_\infty\in H$ and a subsequence of $(x_{n})$, still
denoted by $(x_{n})$, that converges weakly  to $x_\infty$ in $H$.
Since $\Phi$ is closed and convex, it has a continuous affine minorant.
Hence there exist $a\in \R$ and $p\in H$ such that $\Phi(x)\geq a +\langle p,x\rangle$ for every $x\in H$.
By using inequality (\ref{eq.majopsi+phi}), we infer that
$$\Psi(x_{n})
\leq -\eps_n\left[\frac{\eta}{2}+a+ \langle p, x_n\rangle\right].$$
Taking the upper limit when 
$n\to +\infty$, we find 
\begin{equation}\label{eq.limsuppsi}
\limsup_{n\to +\infty}\Psi(x_{n})\leq 0.
\end{equation}
On the other hand, since $\Psi(x_{n})\geq 0$, we infer from
(\ref{eq.majopsi+phi_bis}) that
\begin{equation}\label{eq.limsupphi}
\limsup_{n\to +\infty}\Phi(x_{n})\leq -\frac{\eta}{2}.
\end{equation}
 From the closedness of $\Psi$ (resp. $\Phi$) with respect to the weak topology in $H$ and inequality (\ref{eq.limsuppsi}) (resp. (\ref{eq.limsupphi})), we deduce respectively
that
$$ \Psi(x_\infty)\leq \liminf_{n\to +\infty}\Psi(x_{n})\leq \limsup_{n\to +\infty}\Psi(x_{n})\leq 0,$$
$$\Phi(x_\infty)\leq \liminf_{n\to +\infty}\Phi(x_{n})\leq  \limsup_{n\to +\infty}\Phi(x_{n})\leq -\frac{\eta}{2}. $$ 
The first inequality implies that $x_\infty\in C$ and the second one gives the contradiction.
\end{proof}

By using the duality theory, we are going to prove 
that the behavior of the map $\eps\mapsto \omega(\eps)$ can be interpreted with the conjugates of $\Psi$ and $\Phi$.
Let us first recall the following general theorem, see for example \cite[Theorem 4.1 p. 58]{EkeTem}.
\begin{thrm}\label{th.EkeTem}
Given two normed spaces $V$ and $Y$, let $F:V\to \rinf$ and $G:Y\to \rinf$ be closed convex functions, and let $L\in \cL(V,Y)$. Consider the primal problem
$$\inf_{u\in V}\{F(u)+G(L u)\},\leqno (\cP)$$
and the dual problem
$$\sup_{p^*\in Y^*}\{-F^*(L^*p^*)-G^*(-p^*)\}.\leqno (\cP^*)$$
Then we have $\sup \cP^*\leq \inf \cP$. If moreover $\inf \cP$ is finite and if
there exists $u_0\in \dom F$ such that $G$ is continuous at $L u_0$,
 then $\inf \cP=\sup \cP^*$, and $(\cP^*)$ has at least one solution.
\end{thrm}

\begin{prpstn}\label{pr.majo} Assume Hypotheses $(\cH_\Psi)$-$(\cH_\Phi)$.\\
(a) For every $\eps\geq 0$, we have
\begin{equation}\label{eq.majo_psi*_phi*}
|\omega(\eps)|\leq \inf_{p\in H}\left\{\Psi^*(\eps\, p)+\eps\,\Phi^*(-p)\right\}.
\end{equation}
(b) Let $\eps\geq 0$, assume $\omega(\eps)>-\infty$ and the following qualification condition
\begin{equation}\tag{QC'} \mbox{there exists $x_0\in \dom \Psi$ such that $\Phi$ is continuous at $x_0$,}
\end{equation} 
then we have 
\begin{equation}\label{eq.lien_conjug}
|\omega(\eps)|=\min_{p\in H}\left\{\Psi^*(\eps\, p)+\eps\, \Phi^*(-p)\right\}.
\end{equation}
(c) Assume the qualification condition {\rm(QC)}\footnote{Notice that {\rm(QC)} is slightly stronger than (QC').}
\begin{equation}\tag{QC}\mbox{there exists $x_0\in C$ such that $\Phi$ is continuous at $x_0$.}
\end{equation}
Then there exists $p\in \ran(N_C)$ such that, for every $\eps \geq 0$, 
\begin{equation}\label{eq.majo_psi^*-sigmaC}
|\omega(\eps)|\leq \Psi^*(\eps\, p)-\sigma_C(\eps\, p).
\end{equation}
(d) Assume {\rm(QC)}. Then  Condition \eqref{eq.Attouch-Czarnecki} implies Condition \eqref{eq.cond_beta}.
\end{prpstn}

\begin{proof} $(a)$ Let us apply
Theorem \ref{th.EkeTem} with $V=Y=H$, \, $F=\Psi$, \, $G=\eps\, \Phi$ \, and \, $L={\rm Id}_H$. The primal minimization problem $(\cP_\eps)$ reads as
$$\inf_{x\in H}\left\{\Psi(x)+\eps\, \Phi(x)\right\}.\leqno (\cP_\eps)$$
For every $\eps>0$, the dual problem~is
$$\sup_{p\in H}\left\{-\Psi^*(p)-\eps\, \Phi^*(-p/\eps)\right\}.\leqno (\cP_\eps^*)$$
From the general relation $\sup \cP_\eps\leq\inf\cP_\eps^*$, we deduce that
$$|\omega(\eps)|=-\omega(\eps)\leq\inf_{p\in H}\left\{\Psi^*(p)+\eps\, \Phi^*(-p/\eps)\right\}.$$
Replacing $p$ with $\eps\, p$, we immediately obtain inequality (\ref{eq.majo_psi*_phi*}). This inequality trivially holds true for $\eps=0$, hence it is valid
for every $\eps \geq 0$.

\medskip

$(b)$ Since condition (QC') is satisfied, Theorem \ref{th.EkeTem} shows that
$\inf \cP_\eps=\sup \cP_\eps^*$ and that $(\cP_\eps^*)$ has at least one solution. This implies that
$$|\omega(\eps)|=\min_{p\in H}\left\{\Psi^*( p)+\eps\, \Phi^*(-p/\eps)\right\}.$$
Equality (\ref{eq.lien_conjug}) follows immediately.

\medskip

$(c)$ Given $\xbar\in S=\argmin_C\Phi$, we have $0\in \partial(\Phi+\delta_C)(\xbar)$. The qualification  condition (QC) implies  
$\partial(\Phi+\delta_C)(\xbar)=\partial\Phi(\xbar)+N_C(\xbar)$.
We deduce that $0\in \partial\Phi(\xbar)+N_C(\xbar)$, whence
the existence of $p\in N_C(\xbar)\cap(-\partial\Phi(\xbar))$. For every $\eps\geq 0$, let us write that
\begin{eqnarray*}
\Psi^*( \eps p)+\eps\, \Phi^*(-p)&=&
\left[\Psi^*( \eps p)-\sigma_C(\eps p)\right]
+\eps\,\left[\sigma_C(p)+\delta_C(\xbar)-\langle p,\xbar\rangle\right]\nonumber\\
&&+\eps\,\left[\Phi^*(-p)+\Phi(\xbar)+\langle p,\xbar\rangle\right]. \nonumber
\end{eqnarray*}
Since $p\in N_C(\xbar)$ and $-p\in \partial\Phi(\xbar)$, the Fenchel extremality
relation shows that the second and third brackets are equal to zero. This 
implies that, for every $\eps\geq 0$
$$\Psi^*( \eps p)+\eps\, \Phi^*(-p)=\Psi^*( \eps p)-\sigma_C(\eps p).$$
Inequality (\ref{eq.majo_psi^*-sigmaC}) then immediately follows from (\ref{eq.majo_psi*_phi*}).

\medskip

$(d)$ It follows from $(c)$ and the statement of Conditions \eqref{eq.Attouch-Czarnecki} and \eqref{eq.cond_beta}.\end{proof}

\begin{rmrk}
The qualification condition (QC) may be slightly weakened in the statement of Proposition \ref{pr.majo}, items $(c)$-$(d)$. It suffices to assume that the operator $\partial \Phi + N_C$ is maximal monotone. The same remark applies to the statement of Corollary \ref{co.AttCza}, as was observed in \cite[Theorem 5.1]{AttCza}.
\end{rmrk}

\subsection{Examples}
We now review several examples for which we are able to majorize explicitly the function $\Psi^*-\sigma_C$.
This yields sufficient conditions for \eqref{eq.Attouch-Czarnecki}, and hence for \eqref{eq.cond_beta} in view of Proposition~\ref{pr.majo} $(d)$.

\begin{xmpl}\label{ex.conditioning} 
Let $\Psi:H\to\rinf$ be a closed convex function such that $C=\argmin \Psi\neq\emptyset$. Suppose that for every $x\in H$
$$\Psi(x)\geq \theta(d(x,C)),$$
where the closed convex function $\theta:\R\to \rinf$ is even\footnote{The assumptions on $\theta$ automatically imply that $0\in \argmin \theta$.} and such that
$\theta(0)=0$.
Then we have for every $\eps\geq 0$ and $p\in H$
\begin{equation}\label{eq.majo_Ieps}
\Psi^*(\eps\,p)-\sigma_C(\eps\, p)\leq \theta^*(\eps\, \|p\|).
\end{equation}
\end{xmpl}
\begin{proof}
From a classical result, the conjugate of the function $\theta(d(.,C))$ is the function $\theta^*(\|\,.\,\|)+\sigma_C$, see for example \cite[Exercise IV.17]{Aze}.
It ensues that $\Psi^*\leq \theta^*(\|\,.\,\|)+\sigma_C$, and the conclusion follows immediately.
\end{proof}

Under the assumptions of Example \ref{ex.conditioning}, the key condition (\ref{eq.Attouch-Czarnecki}) of Corollary \ref{co.AttCza} is satisfied if for every $p\in H$,
$$\int_0^{+\infty}\beta(t)\, \theta^*(\|p\|/\beta(t))\,dt<+\infty.$$
\begin{rmrk} Assume that there exists $a>0$ such that $\Psi(x)\geq a\,d(x,C)$ for every $x\in H$. By applying the above proposition with $\theta(t)=a\, |t|$, we find $\Psi^*(\eps\,p)-\sigma_C(\eps\, p)\leq \delta_{[-a,a]}(\eps\,|p|)$, and hence $\Psi^*(\eps\,p)-\sigma_C(\eps\, p)=0$ for $\eps$ small enough. In this case,  condition (\ref{eq.Attouch-Czarnecki}) is automatically satisfied.
\end{rmrk}
\begin{rmrk} \label{re.conditioning} 
Assume that there exist $a>0$ and $r>1$ such that 
\begin{equation}\label{eq.mino_dist}
\Psi(x)\geq a\,d^r(x,C),
\end{equation}
 for every $x\in H$. Let us apply the above proposition with 
the function $\theta:\R\to \R$ defined by $\theta(t)=a\, |t|^r$. Since $(|\,.\,|^r/r)^*=(|\,.\,|^{r^*}/r^*)$, where
 $r^*$ is the conjugate exponent of $r$, \ie $r^*=1/(1-1/r)$, we easily obtain
$$\theta^*(t)= \frac{(ar)^{1-r^*}}{r^*}\, |t|^{r^*}.$$
In view of (\ref{eq.majo_Ieps}), we infer that $\Psi^*(\eps\,p)-\sigma_C(\eps\, p)\leq \frac{(ar)^{1-r^*}}{r^*}\, (\eps\,\|p\|)^{r^*}$.
In this case, condition (\ref{eq.Attouch-Czarnecki}) is satisfied as soon as 
$$\int_0^{+\infty}(1/\beta(t))^{r^*-1}\, dt<+\infty.$$

\end{rmrk}

\begin{xmpl}\label{ex.quadratic}
Let  $L\in \cL(H)$ and let $\Psi:H\to\rinf$ be a closed convex function such that $C=\argmin \Psi=\ker L$. Suppose that $\Psi(x)\geq \demi\|Lx\|^2$ for all $x\in H$.
Then we have for every $\eps \geq 0$ and $p\in \ran(L^*)$,
$$ \Psi^*(\eps\, p)-\sigma_C(\eps\, p)\leq \frac{\eps^2}{2}\, d^2\left(0,(L^*)^{-1}(p)\right).$$
\end{xmpl}
\begin{proof}
By applying Theorem \ref{th.EkeTem}, we can show\footnote{The details are left to the reader.} that the conjugate of the function $x \mapsto \demi\|Lx\|^2$ is given by
$$p\mapsto\left\{
\begin{array}{ll}
\demi d^2\left(0,(L^*)^{-1}(p)\right)& \mbox{ if } \,p\in \ran(L^*)\\
+\infty& \mbox{ if } \,p\notin \ran(L^*).
\end{array}
\right.$$
It ensues that for every $p\in \ran(L^*)$,
\begin{equation}\label{eq.majo_psi*}
\Psi^*(p)\leq \demi d^2\left(0,(L^*)^{-1}(p)\right).
\end{equation}
On the other hand, since the set $\ker L$ is a subspace of $H$, we have
\begin{equation}\label{eq.sigma_ker}
\sigma_{\ker L}=(\delta_{\ker L})^*=\delta_{(\ker L)^\perp}.
\end{equation}
Recalling that $\ran(L^*)\subset (\ker L)^\perp$, we deduce from (\ref{eq.majo_psi*}) and (\ref{eq.sigma_ker}) that for every $\eps\geq 0$ and $p\in \ran(L^*)$,
$$ \Psi^*(\eps\, p)-\sigma_{\ker L}(\eps\, p)\leq \frac{\eps^2}{2} d^2\left(0,(L^*)^{-1}(p)\right).$$
\end{proof}
Under the assumptions of Example \ref{ex.quadratic}, the key condition (\ref{eq.Attouch-Czarnecki}) of Corollary \ref{co.AttCza} is satisfied if
$$\int_0^{+\infty}1/\beta(t)\,dt<+\infty.$$

Recall that for $\rho\geq 0$, the $\rho$-Hausdorff distance between two 
nonempty sets $K$ and $K'$ is defined by
$$\haus_{\rho}(K,K')=\sup_{\|x\|\leq \rho}|d(x,K)-d(x,K')|,$$
see \cite{Beer} for an extended study of this notion.

\begin{xmpl}
Let  $K\subset H$ be a closed convex set such that $0\in K$.
 Define the function $\Psi:H\to \rinf$ by $\Psi=\delta_\BB+\sigma_K$. The set of minima of $\Psi$ is given by $C=\BB\cap N_K(0)$, and we have for every $\eps \geq 0$ and  
$p\in H$,
$$
\Psi^*(\eps\,p)-\sigma_C(\eps\,p)\leq \haus_{\eps\|p\|}(K,T_K(0)).
$$
\end{xmpl}
\begin{proof}
First observe that the assumption $0\in K$ implies that $\sigma_K(x)\geq 0$ for every $x\in H$. We infer that $\Psi(x)\geq 0$ for every $x\in H$, and that
\begin{eqnarray*}
\Psi(x)=0&\Longleftrightarrow& x\in  \BB\, \mbox{ and } \,\sigma_K(x)=0\\
&\Longleftrightarrow& x\in  \BB \, \mbox{ and } \, \langle x,y\rangle\leq 0\, 
\mbox{ for every } y\in K\\
&\Longleftrightarrow& x\in  \BB \, \mbox{ and } \, x\in N_K(0).
\end{eqnarray*}
It ensues that $C=\argmin \Psi=\BB\cap N_K(0)$. Since the Moreau-Rockafellar qualification condition is satisfied, we have
\begin{equation}\label{eq.expression_sigma_C}
\sigma_C=\sigma_{\BB\cap N_K(0)}=\sigma_\BB \,\tdown\,\sigma_{N_K(0)}=\|\,.\,\|\,\tdown\,\delta_{T_K(0)}=d(.,T_K(0)).
\end{equation}
In the same way, we obtain
\begin{equation}\label{eq.expression_Psi*}
\Psi^*=\sigma_\BB \,\tdown\,\delta_K=\|\,.\,\|\,\tdown\,\delta_K=d(.,K).
\end{equation}
In view of (\ref{eq.expression_sigma_C})-(\ref{eq.expression_Psi*}), we deduce that for every $\eps\geq 0$ and $p\in H$
$$\Psi^*(\eps\,p)-\sigma_C(\eps\,p)= d(\eps p, K)-d(\eps p, T_K(0))\leq \haus_{\eps\|p\|}(K,T_K(0)).$$
\end{proof}

\section{Examples of coupled gradient systems with multiscale aspects}\label{se.examples}
\subsection{A two-dimensional example}
Take $H=\R^2$ and fix $a>0$. Consider the function $\Psi:\R^2\to \rinf$ defined by
$$\Psi(x,y)=\left\{\begin{array}{cl}
\frac{y^2}{2(a^2-x^2)}& \mbox{if}\quad (x,y)\in ]-a,a[\times \R\\
0&\mbox{if}\quad (x,y)\in\{(-a,0),(a,0)\}\\
+\infty& \mbox{elsewhere.} 
\end{array}
\right.$$
It is easy to check that $\Psi(x,y)=\frac{1}{2a} (\sigma_D(a+x,y)+\sigma_D(a-x,y))$,
where $\sigma_D$ is the support function of the set $D$ defined by
$$D=\{(x,y)\in \R^2, \quad 2x+y^2\leq 0\},$$
see for example \cite[Example 2.38]{RocWet}. The function 
$\Psi$ is closed, convex and satisfies $C=\argmin \Psi=[-a,a]\times~\{0\}$. Let us now fix $b\in ]0,a[$, and define the function
$\Phi:\R^2\to \R$ by 
$$\Phi(x,y)=y+\demi [x-b]_+^2+\demi[x+b]_-^2,$$ for every $(x,y)\in \R^2$. The function $\Phi$ is convex and differentiable on $\R^2$. It can easily be seen that $\min_C\Phi=0$, and that $S=\argmin_C\Phi=[-b,b]\times\{0\}$.
Given a nondecreasing map $\beta:\R_+\to\R_+$ such that $\lim_{t\to+\infty}\beta(t)=+\infty$, we are interested in the asymptotic behavior as $t\to +\infty$ of the following dynamical system
\begin{equation}\label{eq.equa_ex1}
\dot X(t)+ \partial \Phi(X(t))+\beta(t) \partial\Psi(X(t))\ni 0, \qquad \mbox{ with } X(t)=(x(t),y(t)). 
\end{equation}
From Corollary \ref{co.viscosity}$(i)$, we obtain that $\lim_{t\to+\infty}d(X(t), \argmin_C\Phi)=0$.  We let the reader check that for every $\eps> 0$, $(0,-a^2\eps)$ is the unique minimum point of the
function $\Psi+\eps \, \Phi$ over $\R^2$. The corresponding minimal value equals $\omega(\eps)=(\Psi+\eps \, \Phi)(0,-a^2\eps)=-a^2\eps^2/2$. 
Condition (\ref{eq.cond_beta}) of Corollary~\ref{co.viscosity} amounts to 
$$\int_0^{+\infty}1/\beta(t)\, dt<~+\infty.$$
 Under this condition, Corollary~\ref{co.viscosity}$(ii)$
shows that $\lim_{t\to +\infty}(x(t),y(t))=(x_\infty, 0)$, for some $x_\infty\in [-b,b]$.
For every $(x,y)\in ]-a,a[\times \R$, we have
\begin{eqnarray}\label{eq.H_v1}
(\Psi+\eps \, \Phi)(x,y)&-&(\Psi+\eps \, \Phi)(0,-a^2\eps) \nonumber\\
&=&\frac{y^2}{2(a^2-x^2)}
+\eps\,y+\frac{\eps}{2} [x-b]_+^2+\frac{\eps}{2}[x+b]_-^2+\demi a^2\eps^2\nonumber\\
&\geq &\frac{y^2}{2(a^2-x^2)}
+\eps\,y+\demi a^2\eps^2.
\end{eqnarray}
Observe that
\begin{equation}\label{eq.H_v2}
\frac{y^2}{2(a^2-x^2)}+\eps\,y+\demi a^2\eps^2\geq \frac{y^2}{2a^2}+\eps\,y+\demi a^2\eps^2=\frac{1}{2a^2} (y+a^2\eps)^2.
\end{equation}
On the other hand, we have
\begin{equation}\label{eq.H_v3}
\frac{y^2}{2(a^2-x^2)}+\eps\,y+\demi a^2\eps^2=\demi \eps^2x^2+\frac{(y+\eps (a^2- x^2))^2}{2(a^2-x^2)}\geq \demi \eps^2x^2.
\end{equation}
By combining (\ref{eq.H_v1}), (\ref{eq.H_v2}) and (\ref{eq.H_v3}), we find
for   every $(x,y)\in ]-a,a[\times \R$,
$$(\Psi+\eps \, \Phi)(x,y)-(\Psi+\eps \, \Phi)(0,-a^2\eps)\geq
\frac{1}{4}\eps^2x^2+\frac{1}{4a^2}(y+a^2\eps)^2.$$
This inequality trivially holds true if $(x,y)\notin \dom \Psi$ or if 
$(x,y)\in\{(-a,0),(a,0)\}$. We infer that for every $(x,y)\in \R^2$ and every
$\eps\leq 1/a$,
\begin{eqnarray*}
(\Psi+\eps \, \Phi)(x,y)-(\Psi+\eps \, \Phi)(0,-a^2\eps)
&\geq&\frac{\eps^2}{4}\,\left(x^2+(y+a^2\eps)^2\right)\\
&= & \frac{\eps^2}{4}\,\left\|(x,y)-(0,-a^2\eps)\right\|^2.
\end{eqnarray*}
Dividing by $\eps$ and replacing $\eps$ with $1/\beta(t)$, we obtain that for every $X=(x,y)\in \R^2$ and every $t$ large enough,
$$(\beta(t)\Psi+ \Phi)(X)-(\beta(t)\Psi+ \Phi)(\xi(t))\geq \frac{1}{4\,\beta(t)}\,\left\|X-\xi(t)\right\|^2,
$$
with $\xi(t)=(0,-a^2/\beta(t))$. This shows that Assumption $(i)$ of  Theorem \ref{th.selection_short} is satisfied. The optimal path $t\mapsto \xi(t)$ converges toward $(0,0)$ as $t \to +\infty$.
The finite length assumption of Theorem \ref{th.selection_short} is fulfilled because the map $t\mapsto 1/\beta(t)$ tends nonincreasingly toward~$0$. Assumption $(ii)$ of  Theorem \ref{th.selection_short} amounts to $\int_0^{+\infty}1/\beta(t)\, dt=+\infty$.
Under this last condition, Theorem \ref{th.selection_short} shows that $\lim_{t\to +\infty}(x(t),y(t))=(0, 0)$.
 To summarize, we have proved that
 
$\bullet$ if $1/\beta\in L^1(0,+\infty)$, then $\lim_{t\to +\infty}(x(t),y(t))=(x_\infty, 0)$, for some $x_\infty\in [-b,b]$;

$\bullet$ if $1/\beta\notin L^1(0,+\infty)$, then $\lim_{t\to +\infty}(x(t),y(t))=(0, 0)$.

\subsection{An example in PDE theory}
Let $\Omega\subset \R^N$ be a bounded domain with $\cC^1$ boundary. Let us consider the space $H=L^2(\Omega)$ endowed with the scalar product $\langle u, v\rangle_H=\int_\Omega uv$ and the corresponding norm. Let $h\in L^2(\Omega)$ be a given function satisfying $\int_\Omega h=~0$, and let $a$, $b\in \R$ be such that
$a\leq b$. Take

\medskip

\noindent $\bullet$ $\Psi:L^2(\Omega)\to\rinf$ defined by $\Psi(u)=\demi \int_\Omega \|\nabla u\|^2-\int_\Omega h u$  if $u\in H^1(\Omega)$ and $\Psi(u)=+\infty$ otherwise.\\
$\bullet$ $\Phi:L^2(\Omega)\to\R$ defined by $\Phi(u)=\demi \int_\Omega\left\{
[u(x)-b]_+^2+[a-u(x)]_+^2\right\}dx$ for every $u\in L^2(\Omega)$. \\
The function $\Psi$ is closed and convex. It is immediate to check that the variational formulation of $\xi\in \partial
\Psi(u)$ is given by
\begin{equation}\label{eq.variational}
\forall v\in H^1(\Omega), \quad \int_\Omega \xi\,v=\int_\Omega \nabla u.\nabla v-\int_\Omega h\,v.
\end{equation}
The function $\Phi$ is convex, differentiable and satisfies $\nabla \Phi(u)=[u-b]_+-[a-u]_+$ for every $u\in L^2(\Omega)$. 
Given a map $\beta:\R_+\to\R_+$ such that $\lim_{t\to+\infty}\beta(t)=+\infty$, we are interested in the asymptotic behavior as $t\to +\infty$ of the following dynamical system
$$
\dot u(t)+ \partial \Phi(u(t))+\beta(t) \partial\Psi(u(t))\ni 0. 
$$
If $u(.)$ is a solution of the above differential inclusion, then for almost every $t\geq 0$, there exists $\xi(t)\in \partial \Psi(u(t))$ such that
$$\dot u(t)+ [u(t)-b]_+-[a-u(t)]_+ +\beta(t)\xi(t)=0.$$
Taking the scalar product with $v\in H^1(\Omega)$, we obtain
in view of (\ref{eq.variational})
$$\int_\Omega \dot u(t)\, v+\int_\Omega \left( [u(t)-b]_+-[a-u(t)]_+ \right) v\,+\beta(t)\left[\int_\Omega \nabla u(t).\nabla v-\int_\Omega h\,v\right]=0.$$
By using Green's formula, we find for every $v\in H^1(\Omega)$,
\begin{eqnarray*}
\int_\Omega \dot u(t)\, v &+&\int_\Omega \left( [u(t)-b]_+-[a-u(t)]_+ \right) v \nonumber\\
&+&\beta(t)\left[-\int_\Omega \Delta u(t)\, v+\int_{\partial \Omega}\frac{\partial u(t)}{\partial n}\,v-\int_\Omega h\,v\right]=0.
\end{eqnarray*}
This yields
$$
\left\{
\begin{array}{rcl}
\dot u(t)+[u(t)-b]_+-[a-u(t)]_+ +\beta(t)\left[-\Delta u(t)-h\right]&=0& \mbox{ on } \Omega,\\
\frac{\partial u(t)}{\partial n}&=0& \mbox{ on } \partial \Omega.
\end{array}
\right.
$$
The elements of $C=\argmin \Psi$ are solutions of the minimization problem
$$\inf\left\{\demi\int_\Omega\|\nabla u\|^2-\int_\Omega h\,u: \quad u\in H^1(\Omega)\right\}.$$
The corresponding weak variational formulation is given by
\begin{equation}\label{eq.weak_variational}
\forall v\in H^1(\Omega), \quad \int_\Omega\nabla u.\nabla v=\int_\Omega h\, v.
\end{equation}
Since $\int_\Omega h=0$, it is well-known that such solutions exist, and they satisfy the following Neumann boundary value problem
$$\left\{\begin{array}{rcl}
-\Delta u-h&=0& \mbox{ on } \Omega,\\
\frac{\partial u}{\partial n}&=0& \mbox{ on } \partial \Omega.
\end{array}
\right.$$
Denoting by $\hu$ a particular solution, the set $C=\argmin \Psi$ is the straight line $C=\{\hu+m, \, m\in \R\}$. 
Let us now check that the function $\Psi$ satisfies the inf-compactness property (\ref{eq.inf-compact}). Given $R>0$ and $l\in \R$, let $u\in L^2(\Omega)$ be in the lower level set
$$\Lambda_{R,l}=\{u\in L^2(\Omega),\quad \|u\|_{L^2}\leq R, \,\Psi(u)\leq l\}.$$
From the definition of $\Psi$, we have $u\in H^1(\Omega)$ and
\begin{eqnarray*}
\int_\Omega\|\nabla u\|^2&\leq &2l+2\,\int_\Omega h\,u\\
&\leq&2l+2\,\|h\|_{L^2}\|u\|_{L^2}\leq 2\,l+2\, R\,\|h\|_{L^2}.
\end{eqnarray*}
We immediately deduce that
$$\|u\|^2_{H^1}=\int_\Omega u^2+\int_\Omega\|\nabla u\|^2\leq R^2+2\,l+2\, R\,\|h\|_{L^2},$$
which shows that the set $\Lambda_{R,l}$ is bounded for the $H^1(\Omega)$-norm.
Since $\Omega$ is bounded with $\cC^1$ boundary,  by the Rellich-Kondrachov theorem, the injection $H^1(\Omega)\hookrightarrow L^2(\Omega)$ is compact. We conclude that $\Lambda_{R,l}$ is relatively compact for the $L^2(\Omega)$-norm, hence the function $\Psi$ satisfies the inf-compactness property (\ref{eq.inf-compact}). \\
Let us now determine the set $S=\argmin_C\Phi$. Since the function $\Phi$ is continuous, convex and coercive, the set $\argmin_C\Phi$ is a nonempty segment included in~$C$.
Recall that $u\in \argmin_C\Phi$ if and only if it satisfies
the optimality condition $-\nabla \Phi(u)\in N_C(u)$. Since the set $C$ is a straight line directed by the space of constant functions, it is clear that
$$N_C(u)=\left\{p\in L^2(\Omega), \quad \langle p,1\rangle_{L^2(\Omega)}=0\right\}=
\left\{p\in L^2(\Omega), \quad \int_\Omega p=0\right\}.$$
Finally, we obtain the equivalences
\begin{eqnarray}
u\in \argmin_C\Phi&\Longleftrightarrow&\int_\Omega \nabla \Phi(u)(x)\, dx=0\nonumber\\
&\Longleftrightarrow&\int_\Omega \Big([u(x)-b]_+-[a-u(x)]_+\Big)\, dx=0.\label{eq.cond_opt}
\end{eqnarray}
Assuming that $\hu\in \argmin_C\Phi$, let us denote by $\inf_\Omega \hu$ (resp. $\sup_\Omega \hu$) the essential
infimum (resp. supremum) of $\hu$ over the set $\Omega$.
We distinguish the cases \, $\sup_\Omega \hu-\inf_\Omega \hu> b-a$, and \,$\sup_\Omega \hu-\inf_\Omega \hu\leq b-a$.

\medskip

\noindent {\it Case 1:} \, $\sup_\Omega \hu-\inf_\Omega \hu> b-a$. In view of condition (\ref{eq.cond_opt}), we deduce that the sets
 $$\Omega_+=\{x\in \Omega, \quad \hu(x)>b\}\quad \mbox{ and } \quad
\Omega_-=\{x\in \Omega, \quad \hu(x)<a\}$$
have positive measures.
For $m\in \R$, let us define the quantity $\theta(m)$ by
$$\theta(m)=\int_\Omega \Big([\hu(x)+m-b]_+-[a-m-\hu(x)]_+\Big)\, dx.$$
Recalling that $\theta(0)=0$, we have for every $m\geq 0$
\begin{eqnarray*}
\theta(m)&=&\int_\Omega \Big([\hu(x)+m-b]_+-[\hu(x)-b]_+\Big)\, dx \\
&+&\int_\Omega \Big([a-\hu(x)]_+-[a-m-\hu(x)]_+\Big)\, dx\\
&\geq& \int_\Omega \Big([\hu(x)+m-b]_+-[\hu(x)-b]_+\Big)\, dx\\
&\geq&\int_{\Omega_+} \Big([\hu(x)+m-b]_+-[\hu(x)-b]_+\Big)\, dx\\
&=&\int_{\Omega_+} m\, dx=m\, |\Omega_+|.
\end{eqnarray*}
In the same way, we obtain $\theta(m)\leq m\, |\Omega_-|$ for every $m\leq 0$.
Since $|\Omega_+|$ and $|\Omega_-|$ are positive, this implies that $\theta(m)=0$ if and only if $m=0$. In view of (\ref{eq.cond_opt}), we conclude that $\hu$ is the unique minimum of $\Phi$ over the set $C=\{\hu+m, \, m\in \R\}$. We then infer from Corollary \ref{co.viscosity}$(i)$ that $\lim_{t\to +\infty} u(t)=\hu $ strongly
in $L^2(\Omega)$. 

\medskip

\noindent {\it Case 2:} \, $\sup_\Omega \hu-\inf_\Omega \hu\leq b-a$. In view of condition (\ref{eq.cond_opt}), we deduce that $\hu(x)\in [a,b]$ for almost every $x\in \Omega$.
We then have $\Phi(\hu)=0$, hence $\hu\in \argmin \Phi$. It ensues that
\begin{eqnarray*}
S&=& \argmin \Psi\cap \argmin \Phi\nonumber\\
&=&\left\{\hu+m, \quad m\in\left[a-\inf_\Omega \hu, b-\sup_\Omega \hu\right]\right\}.
\end{eqnarray*}
 By combining Corollary \ref{co.inter_nonvide} and Remark \ref{rk.strong_cv}, we deduce
that there exists $\bar{u}\in S$ such that $\lim_{t\to +\infty} u(t)=\bar{u}$ strongly in $L^2(\Omega)$. 

\medskip

In fact the convergence is strong in $H^1(\Omega)$ in each of the above cases. Indeed, observe that
\begin{eqnarray*}
\|u(t)-\bar{u}\|^2_{H^1}&=&\int_\Omega \|\nabla u(t)-\nabla \bar{u}\|^2
+\int_\Omega |u(t)-\bar{u}|^2\\
&=&\int_\Omega \|\nabla u(t)\|^2-2\, \int_\Omega \nabla u(t)\nabla \bar{u}+\int_\Omega\|\nabla \bar{u}\|^2 +\int_\Omega |u(t)-\bar{u}|^2.
\end{eqnarray*}
By using the weak variational formulation (\ref{eq.weak_variational}), we obtain that $\int_\Omega \nabla u(t)\nabla \bar{u}=\int_\Omega h\,u(t)$, and that
$\int_\Omega \|\nabla \bar{u}\|^2=\int_\Omega h\,\bar{u}$.
We immediately deduce from the above equality that
$$\|u(t)-\bar{u}\|^2_{H^1}=2\, \big(\Psi(u(t))-\Psi(\bar{u})\big)+\int_\Omega |u(t)-\bar{u}|^2.$$
Since $\lim_{t\to+\infty} \Psi(u(t))=\min_H \Psi$ (see Remark \ref{rk.strong_cv}) and since
$\lim_{t\to+\infty} \|u(t)-\bar{u}\|_{L^2}=~0$, we conclude that $\lim_{t\to+\infty} \|u(t)-\bar{u}\|_{H^1}=0$.



\begin{thebibliography}{99}
\bibitem{AlvCab} F. Alvarez, A. Cabot, Asymptotic selection of viscosity 
equilibria of semilinear evolution equations by the introduction of a slowly vanishing term, {\em Discrete and Continuous Dynamical Systems}, 15 (2006), pp. 921-938.


\bibitem{Att} H. Attouch, Variational convergence for functions and operators, {\em Applicable Mathematics Series}, Pitman Advanced Publishing Program, London, (1984).

\bibitem{Att2} H. Attouch, Viscosity solutions of minimization problems,
{\em SIAM J. Optim.}, 6 (3), (1996),  pp. 769--806. 

\bibitem{AttButMic} H. Attouch, G. Buttazzo, G. Michaille, Variational analysis in Sobolev and BV spaces. Applications to PDE's and Optimization, Second edition, {\em MOS/SIAM Series on Optimization}, MO17, Society for Industrial and Applied Mathematics (SIAM), Philadelphia, PA, (2014).

\bibitem{AttCom} H. Attouch and R. Cominetti, A dynamical approach to convex
minimization coupling approximation with the steepest descent method,  {\em J.
Differential Equations}, 128 (2), (1996), 
pp. 519-540.
%
\bibitem{AttCza} H. Attouch, M.-O. Czarnecki, Asymptotic behavior of coupled dynamical systems with multiscale aspects, {\em J. Differential Equations}, 248 (2010), pp. 1315-1344.

\bibitem{AttCzaPey1} H. Attouch, M.-O. Czarnecki, J. Peypouquet, Prox-penalization and splitting methods for constrained variational problems, {\em SIAM Journal on Optimization}, 21 (2011), pp. 149-173. 

\bibitem{AttCzaPey2} H. Attouch, M.-O. Czarnecki, J. Peypouquet,  Coupling forward-backward with penalty schemes and parallel splitting for constrained variational inequalities, {\em SIAM Journal on Optimization}, 21 (2011), pp. 1251-1274.
%
\bibitem{AD} H. Attouch, A. Damlamian,  Strong solutions for parabolic variational inequalities,
 {\em  Nonlinear Analysis, TMA}, 2 (1978), No. 3,  pp. 329-353.  
   
\bibitem{Aze} D. Az\'e, El\'ements d'analyse convexe et variationnelle, Ellipses, Paris, (1997).
     
\bibitem{BC} 
H. H. Bauschke and P. L. Combettes,
{\em Convex analysis and monotone operator theory in Hilbert spaces,}
Springer, New York, (2011).
%
\bibitem{BaiCom}  J.-B. Baillon, R. Cominetti, A convergence result for non-autonomous subgradient evolution equations and its application to the steepest descent exponential penalty trajectory in linear programming, {\em J. Funct. Anal.}, 187 (2001), pp. 263-273.
%
\bibitem{BaiBre} J.B. Baillon, H. Br\'ezis, Une remarque sur le comportement asymptotique des semi-groupes non lin\'eaires, {\em Houston J. Math.}, 2
(1976), pp. 5-7.

\bibitem{BauLarSen} H.H. Bauschke, D.A. McLaren, H.S. Sendov, Fitzpatrick functions: inequalities, examples, and remarks on a problem by S. Fitzpatrick, {\em Journal of Convex Analysis}, 13 (2006), pp. 499-523.

\bibitem{Beer} G. Beer, {\em Topologies on closed and closed convex sets}, Mathematics and its applications, Kluwer Academic Publishers, Vol 268, 
(1993).

\bibitem{BotCse} R. I. Bot, E. R. Csetnek, Forward-Backward and Tseng's type penalty schemes for monotone inclusion problems, {\em Set-Valued and Variational Analysis}, 22 (2014), pp. 313-331.


\bibitem{Bre} H. Br\'ezis, Op\'erateurs maximaux monotones dans les espaces de Hilbert et \'equations d'\'evolution, {\em Lecture Notes} vol. 5, North-Holland, 1972.

\bibitem{Bre78} H. Br\'ezis, Asymptotic behavior of some evolution systems, {\em Nonlinear Evolution Equations}, Academic Press, New York (1978), pp. 141-154.
\bibitem{BreHar} H. Br\'ezis, A. Haraux, Image d'une somme d'op\'erateurs monotones et applications, {\em Israel J. Math.} 23 (1976), pp. 165-186.
\bibitem{Bru} R.E. Bruck, Asymptotic convergence of nonlinear contraction semigroups in Hilbert spaces,  {\em J. Funct. Anal.}, 18 (1975), pp. 15-26.

\bibitem{BurSva} R.S. Burachik, B.F. Svaiter, Maximal monotone operators, convex functions and a special family of enlargements, {\em Set-Valued Analysis}, 10 (2002), pp. 297-316.
%
\bibitem{Cab} A. Cabot, Proximal point algorithm controlled by a slowly vanishing term: Applications to hierarchical minimization, {\em SIAM Journal on Optimization}, 15 (2005), No 2, pp. 555-572.
%
\bibitem{ComPeySor} R. Cominetti, J. Peypouquet, S. Sorin, Strong asymptotic convergence of evolution equations governed by maximal monotone operators with Tikhonov regularization, {\em J. Differential Equations}, 245 (2008), pp. 3753-3763.
%
\bibitem{EkeTem} I. Ekeland, R. Temam, Convex Analysis and Variational Problems,  SIAM Classics in Applied Mathematics, 28, (1999).
%
\bibitem{Fitz} S. Fitzpatrick, Representing monotone operators by convex functions, {\em Workshop Miniconference on Functional Analysis and Optimization (Canberra, 1988)}, Proceedings of the Centre for Mathematical Analysis, Australian National University vol. 20, Canberra, Australia, pp. 59-65, (1988).
%
\bibitem{FurMiyKen} H. Furuya, K. Miyashiba, N. Kenmochi, Asymptotic behavior of solutions to a class of nonlinear evolution equations, {\em J. Differential Equations}, 62 (1986), pp. 73-94.
%
\bibitem{Ken}  N. Kenmochi, Solvability of nonlinear  equations with time-dependent constraints and applications, {\em Bull. Fac. Educ. Chiba Univ.}, 30 (1981), pp. 1-87.
%
\bibitem{Lem} B. Lemaire, On the convergence of some iterative methods for convex minimization, recent developments in optimization, Edited by R. Durier and C. Michelot,
{\em Lecture Notes in Economics and Mathematical Systems}, Springer-Verlag, Berlin, Germany, 429 (1995), pp. 252-268.

\bibitem{MarSva} J.-E. Mart\'inez-Legaz, B.F. Svaiter, Monotone operators representable by l.s.c. convex functions,
{\em Set-Valued Analysis}, 13 (2005), pp. 21-46.

\bibitem{MarThe} J.-E. Mart\'inez-Legaz, M. Th\'era, A convex representation of maximal monotone operators,
{\em J. of Nonlinear and Convex Anal.}, 2 (2001), pp. 243-247.
%
\bibitem{Opi} Z. Opial, Weak convergence of the sequence of successive approximations for nonexpansive mappings, {\em Bull. Amer. Math. Soc.}, 73 (1967), pp. 591-597.
%
\bibitem{Pas} G. B. Passty, Ergodic convergence to a zero of the sum of monotone operators in Hilbert spaces, {\em Journal of Mathematical Analysis and Applications}, 72 (1979), pp. 383-390.
%
\bibitem{PenZal} J.P. Penot, C. Zalinescu, On the convergence of maximal monotone operators, {\em Proc. Amer. Math. Soc.}, 134 (2005), pp. 1937-1946. 
%
\bibitem{Roc} R.T. Rockafellar, {\em Convex Analysis}, Princeton University Press, Princeton, (1970).
%
\bibitem{RocWet}  R.T. Rockafellar, R. Wets, {\em Variational analysis},
Springer, Berlin, (1998).
%
\bibitem{SimZal1} S. Simons, C. Zalinescu, A new proof for Rockafellar's characterization of maximal monotone operators,
{\em Proc. Amer. Math. Soc.}, 132 (2004), pp. 2969-2972.
%
\bibitem{SimZal2} S. Simons, C. Zalinescu, Fenchel duality, Fitzpatrick functions and maximal monotonicity,
{\em J. of Nonlinear and Convex Anal.}, 6 (2005), pp. 1-22.
%

\bibitem{Tor2} D. Torralba,, D\'eveloppements asymptotiques pour les m\'ethodes d'approximation par viscosit\'e, {\em C.R. Acad. Sci. Paris S\'er. I Math.,} 322 (1996), pp. 123-128.

\end{thebibliography}
\end{document}